\documentclass[a4paper,leqno,twoside]{amsart}%
\usepackage{amsmath,amssymb,amsthm,enumerate,hyperref,color}
\numberwithin{equation}{section}
\newtheorem{theorem}{Theorem}[section]

\newtheorem{proposition}[theorem]{Proposition}
\newtheorem{lemma}[theorem]{Lemma}
\newtheorem{corollary}[theorem]{Corollary}
\newtheorem{remark}[theorem]{Remark}

\newcommand{\nor}{\Arrowvert}

\newcommand{\rad}{{\text{\upshape rad}}}

\def\ep{\e_p}

\def\e{{\varepsilon}}

\def\d{\delta}
\def\g{{\gamma}}

\def\L{{\Lambda}}

\def\a{{\alpha}}

\newcommand{\R}{\mathbb{R}}
\newcommand{\N}{\mathbb{N}}

\newcommand{\loc}{\mathop{\mathrm{loc}}}
\newcommand{\nod}{\mathop{\mathrm{nod}}}

\newif\ifcomment \commentfalse
\def\commentON{\commenttrue}

\long\outer\def\BC#1\EC{\ifcomment \sloppy \par \# \ldots\dotfill
{\em #1} \dotfill \# \par \fi } \commentON

\newcommand{\remove}[1]{}

\def\sideremark#1{\ifvmode\leavevmode\fi\vadjust{\vbox to0pt{\vss
 \hbox to 0pt{\hskip\hsize\hskip1em
 \vbox{\hsize2.1cm\tiny\raggedright\pretolerance10000
  \noindent #1\hfill}\hss}\vbox to15pt{\vfil}\vss}}}%
\newcommand{\edz}[1]{\sideremark{#1}}

\definecolor{cadmiumgreen}{rgb}{0.0, 0.42, 0.24}
\definecolor{darkgreen}{rgb}{0.0, 0.5, 0.2}
\definecolor{purple}{rgb}{0.5, 0.0, 0.5}
\newcommand{\AL}{\color{purple}}
\newcommand{\F}{\color{magenta}}
\newcommand{\taglia}{\color{cyan}}

\title[H\'enon problem]{The H\'enon problem with large exponent in the disc}
\author[A.~L.~Amadori, F.~Gladiali]{Anna Lisa Amadori$^\dag$,  Francesca Gladiali$^\ddag$}
\thanks{This work was supported by Gruppo Nazionale per l'Analisi Matematica, la Probabilit\`a e le loro Applicazioni (GNAMPA) of the Istituto Nazionale di Alta Matematica (INdAM). The second author is supported by Prin-2015KB9WPT and Fabbr}
\date{\today}
\address{$\dag$ Dipartimento di Scienze Applicate, Universit\`a di Napoli ``Parthenope", Centro Direzionale di Napoli, Isola C4, 80143 Napoli, Italy. \texttt{annalisa.amadori@uniparthenope.it}}
\address{$\ddag$ Dipartimento di Chimica e Farmacia, Universit\`a di Sassari, via Piandanna 4, 07100 Sassari, Italy. \texttt{fgladiali@uniss.it}}

\begin{document}
\maketitle 

\begin{abstract}
In this paper we consider the H\'enon problem in the unit disc with Dirichlet boundary conditions. We study the asymptotic profile of least energy and nodal least energy radial solutions and then deduce the exact computation of their Morse index for large values of the exponent $p$. 
As a consequence of this computation a multiplicity result for positive and nodal solutions is obtained.

\end{abstract}

{\bf Keywords: } nodal solutions; concentration phenomena; Morse index; least energy, radial and non-radial solutions.

{\bf AMS Subject Classifications:}  35J91, 35B06, 35B40, 35P05

\tableofcontents

\section{Introduction}\label{sec:intro}
In this paper we study problem
\begin{equation}\label{H}
\left\{\begin{array}{ll}
-\Delta u = |x|^{\alpha}|u|^{p-1} u \qquad & \text{ in } B, \\
u= 0 & \text{ on } \partial B,
\end{array} \right.
\end{equation}
where $\a\geq 0$, $p>1$ and 
$B$ stands for the unit ball of the plane. 
When $\a>0$ problem \eqref{H} is known as the H\'enon problem since it has been introduced by H\'enon in \cite{H} in the study of stellar clusters in radially symmetric settings, in 1973.
For $\a=0$ problem \eqref{H} coincides with the Lane-Emden problem
\begin{equation}\label{LE}\left\{\begin{array}{ll}
-\Delta u = |u|^{p-1} u \qquad & \text{ in } B, \\
u= 0 & \text{ on } \partial B,
\end{array} \right.
\end{equation}
and we will see that the connections between \eqref{H} and \eqref{LE} are deeper.  Indeed radial solutions to \eqref{H} can be viewed as radial solutions to \eqref{LE} in a sense which will be clarified in Section \ref{sec:2}.
\\
The appeal of the $\a>0$ case is due to various aspects. First, when the dimension $N$ of the space is $N\geq 3$, 
\eqref{H} admits solutions also in the supercritical range of $p$, as observed by Ni in \cite{Ni}, where another {\em critical} exponent has been shown in the radial framework.  The second main reason of interest is the symmetry breaking phenomenon due to the term $|x|^\a$ which allows the coexistence of radial and nonradial solutions also in the case when they are positive. It is known indeed that also solutions which minimize the energy are not radial when $\a$ is large enough, see \cite{SSW}.
\\ 
 In this paper, carrying on the study of the H\'enon problem started in \cite{AG-sez2,AG-N3}, we consider \eqref{H} in the unit disc, where it admits solutions for every $p>1$ and no critical exponent appears
and in particular we focus on large values of $p$, where concentration phenomena take place and nonradial positive solutions arise. Indeed it has been shown in \cite{AdG} and \cite{GGP14}, both dealing with the Lane Emden problem, that radial solutions behave like a spike. Such kind of concentration differs from the one occurring in the higher dimensional case when $p$ approaches the critical exponent: firstly because solutions stay bounded, secondly because the concentration of nodal solutions follows different paths (and has different limit problems), depending on the nodal zone which is focused.
In this context then we analyze the asymptotic behaviour of solutions to \eqref{H} for large values of $p$, starting from the radial ones, both positive and sign changing. What we obtain is that, in the radial setting, the concentration phenomena known for the Lane Emden case ($\a=0$) in the plane extend to solutions to \eqref{H}, strengthening the connections between the two problems. In particular the limit of some of the parameters coincide as the effect of the large exponent $p$ removes the influence of the term $|x|^\a$, even if the concentration takes place at $x=0$ where its effects are usually higher. Nevertheless this term plays a significant role that shows in the asymptotic profiles of the solutions, affects their Morse index and  produces  nonradial solutions, also even positive ones. 

\ 

Let us present which type of solutions we are interested in and the main results.
 
Since we are in the plane and $H^1_0(B)$ is compactly embedded in $L^p(B)$ for every $p$, problem \eqref{H}  admits solutions for every value of $p>1$ and $\a\geq 0$. Solutions can be found minimizing the 
the Energy functional
\begin{equation}\label{eq:energy}
\mathcal E(u)=\frac 12 \int _B |\nabla u|^2-\frac 1{p+1}\int_B |x|^\a |u|^{p+1}\end{equation} 
constraint on the Nehari manifold
\[\mathcal N=\{v\in H^1_0(B)\ : \ \int_B|\nabla v|^2=\int_B |x|^\a |v|^{p+1}\}.\]
The solutions produced in this way are positive in $B$ and  are called least energy solutions.
\\
Nodal solutions can be obtained instead minimizing $\mathcal{E}(u)$ on the
nodal Nehari manifold
\[ \begin{array}{rl}\mathcal N_{\nod}=\Big\{v\in H^1_0(B) :&   v^+ \neq 0, \ \int_B|\nabla v^+|^2=\int_B |x|^\a |v^+|^{p+1}, \\ &  v^-\neq 0 , \ \int_B|\nabla v^-|^2=\int_B |x|^\a |v^-|^{p+1}\Big\} .\end{array} \]
 that has been introduced in \cite{CCN} and \cite{BW}.
Here $s^+$ ($s^-$) stands for  the positive (negative) part of $s$.
Minima on $\mathcal N_{\nod}$ have the least energy among nodal solutions to \eqref{H}
and are called least energy nodal solutions.  By a result in \cite{BW} they have two nodal regions, which are the connected components of the set $\{x\in B \ : u(x)\neq 0\}$. 
Moreover in \cite{BWW} it has been proved that least energy solutions partially inherit the symmetries of the domain, being foliated Schwarz symmetric, namely axially symmetric with respect to an axis passing through the origin and nonincreasing in the polar angle from this axis (see also \cite{PacellaWeth}).

Let us recall that the Morse index of a solution $u$ is the maximal dimension of a subspace $X\subseteq H^1_0(B)$ where the quadratic form 
	\begin{equation}\label{eq:Q-u}
	Q_u(\psi):=\int_B|\nabla \psi|^2-p|x|^\a|u|^{p-1}\psi^2\, dx\end{equation}
	is negative defined.
	 Since from one side $\mathcal E(u)$ is not bounded, and from the other side $Q_u(\psi)= \langle\mathcal E''_u\psi,\psi\rangle$  where $\mathcal E''$ denotes the second Fr\'echet derivative of $\mathcal E$ and $\langle,\rangle$ is the pairing,
		one can see that  the Morse index measures, in a sense,  how a critical point of $\mathcal E$  fails to be a minimum.
Indeed  the least energy and the least energy nodal solutions have Morse index $1$ and $2$ because they are constrained minima on manifolds of codimension $1$ and $2$, respectively, see \cite{BW}. 
	\\  From a different perspective the quadratic form $Q_u$ is associated with the linearized operator at $u$
	\[L_u(\psi):=-\Delta \psi-p |x|^\a|u|^{p-1}\psi\]
	with Dirichlet boundary conditions and  the Morse index can be computed counting (with multiplicity) the negative eigenvalues of $L_u$ in $H^1_0(B)$, but also some negative singular eigenvalues. This equivalence and the characterization of Morse index in terms of the singular eigenvalues of $L_u$	is given in details in \cite{AG-sez2} and will be essential for our aims.

The  aforementioned minimization procedure can be done, 
in principle, in any subspace of $H^1_0(B)$, and particular attention has been devoted to the one of radial functions $H^1_{0,\rad}(B)$. Restricting the energy functional and the (nodal) Nehari manifold to the space $H^1_{0,\rad}(B)$ of radial functions, we end with a least-energy positive radial solution or with a least-energy nodal radial solution to \eqref{H} that we denote respectively by $u^1_p$ and $u^2_p$ highlighting the number of nodal domains. Again the radial Morse index  of $u^1_p$ and $u^2_p$ is respectively $1$ and $2$, 
where by radial Morse index we mean the number of the negative radial eigenvalues of $L_u$, namely eigenvalues which are associated with a radial eigenfunction. But the Morse index of $u^1_p$ and $u^2_p$, depending on $p$ and on $\a$ can be larger,  implying that the least energy solutions are not radial.
Indeed this is the case at least for the nodal solution $u_p^2$ since it has been proved in \cite{AP}  for $\a=0$ and in \cite{AG-sez2} for $\a>0$ that $m(u_p^2)\ge 4 + [\a/2]$ for every $p$ (here $[\cdot]$ stands for the integer part). In \cite{LWZ}
instead it has been shown that  the Morse index of any radial solution to \eqref{H} diverges as $\a\to \infty$ 
and this implies that least energy solutions are nonradial when $\a$ is large enough.

In this paper we analyze problem \eqref{H} as $p\to \infty$, finding the asymptotic profile of radial least-energy solutions $u^1_p$ and $u^2_p$
and then compute the exact Morse index of these solutions, depending on $\a$, for sufficiently large values of the exponent $p$. Next we will see that the knowledge of the Morse index allows to distinguish between different solutions to \eqref{H}, that can be produced by minimizing the energy on the (nodal) Nehari manifold in some other subspace of $H^1_0(B)$.

\

The paper is organized as follows.
	Section \ref{sec:2} is devoted to the asymptotic of the radial solutions, which can be deduced without too much effort from the analysis carried out for the Lane-Emden problem in \cite{AdG} (concerning positive solution) and \cite{GGP14} (concerning nodal least energy solution).

For large values of $p$ problem \eqref{H} is linked to the weighted Liouville problem
\begin{equation}\label{H-L}
\left\{\begin{array}{l}
-\Delta  U=|x|^\a e^U\qquad \ \text{ in } \R^2, \\
\int_{\R^2}|x|^{\a} e^U\, dx<\infty
\end{array} \right.
\end{equation} 
and to the family of its radial solutions described by 
\begin{equation}\label{eq:U-alpha}
U_{\a;\d}(x)=\log \frac {2(2+\a)^2\delta} { (\delta+|x|^{2+\a})^2}  ,\quad \d>0 .
\end{equation}
Imposing the condition $U(0)=0$ selects uniquely the parameter $\d$ (and so the solution to \eqref{H-L}) as 
\begin{equation}\label{eq:delta-1}
\d(\a)=2(2+\a)^2.
\end{equation}
As enlightened in \cite{GGP14}, when describing the asymptotic behaviour of nodal solutions also a singular Liouville problem arises. In the present case it is
a singular version of  problem \eqref{H-L}, precisely
	\begin{equation}\label{S-H-L}
	\left\{\begin{array}{l}
	-\Delta  U=|x|^\a e^U  -(2+\a)\pi \g \d_0\qquad \ \text{ in } \R^2, \\
	\int_{\R^2}|x|^{\a} e^U\, dx<\infty
	\end{array} \right.
	\end{equation} 
	where $\d_0$ denotes the Dirac measure supported at $x=0$.
	A family of radial solutions is given by
	\begin{align}\label{Z-a-g}
	Z_{\a,\g;\d}(x)& =\log\frac{2\left(\frac{(2+\a)(2+\g)}{2}\right)^2\delta |x|^{\frac{2+\a}{2}\g}}{ \left(\delta+|x|^{\frac{(2+\a)(2+\g)}{2}}\right)^2}  & \\ \nonumber 
	& = U_{\a+\frac{2+\a}{2}\g;\d}(x)+\frac{2+\a}{2}\g\log|x| ,\quad & \d>0 
	\end{align}
	where $U_{\a+\frac{2+\a}{2}\g;\d}$ is a solution to \eqref{H-L} as defined in \eqref{eq:U-alpha} with $\a$ replaced by $\a+\frac{2+\a}{2}\g$.

In order to state the results on the asymptotic of the solutions we	need some more notations. 
Concerning the minimal energy radial solution $u^1_p$, it is known by ODE arguments that it has only one critical point at $x=0$. We therefore let
\[  \mu_p = \left| u^1_p(0)\right| , \qquad  \rho_p =   (p \, \mu_p^{p-1})^{-\frac 1{2+\a}},\]
and define the rescaling 
\[\widetilde u_{p}( x)=\frac{p(u^1_p(\rho_p x)-u^1_p(0))}{u^1_p(0)} \quad \text{ as }  |x| < \frac{1}{\rho_p}.\]
We shall see that
\begin{theorem}\label{teo:asymptotic-of-u-1}
	Let $\a\geq 0$ be fixed and let $u_p^1$ be  a least energy radial solution to \eqref{H} corresponding to  $\a$. When $p\to \infty$ we have 
	\begin{align}\label{eq:massimi-pcritico-1}
	\mu_{p}  & \to   \sqrt e, \qquad \rho_p\to 0 , \\
	\label{eq:convergenza-riscalata-1}
	\widetilde u_{p}(x)  & \to  U_{\a, \d(\a)}(x)   = \log \frac {4(2+\a)^4} { (2(2+\a)^2+|x|^{2+\a})^2} \qquad   \text{ in } C^1_{\mathrm{loc}}(\R^2)  .
	\end{align}
\end{theorem}

For what concerns the minimal energy nodal radial solution $u^2_p$, we write $u^2_p(x) = u^2_p(r)$ for $r=|x|$, and denote by $r_{p}$ its unique zero in $[0,1)$, so that $A_{1,p}=[0, r_p)$ and $A_{2,p}=(r_p, 1)$ are its nodal zones. It is know by ODE argument that it has two critical points in $[0,1)$: the first one is $0$ while the second one is $\sigma_p\in A_{2,p}$.
We therefore have two extremal values
\begin{align*}
\mu_{1,p} = \left| u^2_p(0)\right| , & \qquad \mu_{2,p} = \left| u^2_p(\sigma_p)\right| ,
\intertext{two scaling parameters}
\rho_p^1 =   (p |u_p^2(0)|^{p-1})^{-\frac 1{2+\a}} , & \qquad \rho_p^2 =   (p |u_p^2(\sigma_p)|^{p-1})^{-\frac 1{2+\a}} , 
\intertext{and two rescaled functions}
\widetilde u_{1,p}(r)= \frac{p(u^2_p(\rho_p^1 r)-u^2_p(0))}{u^2_p(0)}  &  \quad \text{ as } 0\leq r<\frac 1{\rho_p^1}, \\
\widetilde u_{2,p}(r)=\frac{p(u^2_p(\rho_p^2 r)-u^2_p(\sigma_p))}{u^2_p(\sigma_p)}  & \quad   \text{ as } 0\leq r<\frac 1{\rho_p^2}.
\end{align*}

The asymptotic behaviour of $u^2_p$ is described by next Theorem.
\begin{theorem}\label{teo:asymptotic-of-u}
	Let $\a\geq 0$ be fixed and let $u_p^2$ be a  least energy  nodal radial solution to \eqref{H} corresponding to $\a$. When $p\to \infty$ we have
\begin{align} \label{eq:massimi-pcritico}
&  \mu_{1,p}  \to   \frac {\sqrt e}{\bar t}e^{\frac {\bar t}{2(\bar t +\sqrt e)}}
\approx 2.46  , \qquad       \mu_{2,p}\to e^{\frac {\bar t}{2(\bar t +\sqrt e)}}
\approx 1.17,  \ & \rho_{p}^i\to 0  \ \text{ as } i=1,2
\intertext{ where $\bar t\approx 0.7875$ is the unique root of the equation $2\sqrt e \log t+t=0$, and }
\label{eq:convergenza-riscalata}
& \widetilde u_{1,p} (x)  \to U_{\a; \d(\a)} (x) = \log \frac {4(2+\a)^4} { (2(2+\a)^2+|x|^{2+\a})^2}   &  \text{ in } C^1_{\mathrm{loc}}(\R^2) . 
\end{align}
Moreover
\begin{align}
\label{zeri-pcritico} 	r_{p}  \to 0 , 	\qquad  \frac{r_p}{\rho_p^1 }  \to \infty , \qquad &  \frac{r_p}{\rho_p^2 }  \to 0 , \\
\label{ell-alpha} \sigma_{p} \to 0  , \qquad  \phantom{ \frac{r_p}{\rho_p^1 }  \to \infty , \qquad } &  \frac{\sigma_p}{\rho_p^2}  \to  \left(\frac{2}{2+\a} \ell\right)^{\frac2{2+\a}} 
\end{align}
where $\ell$ is a fixed number, $\ell\simeq 7.1979$.
Starting from $\ell$ we define 
\begin{equation}\label{g-d-ell}
\g: = \sqrt{4+2\ell^2} - 2 \simeq 8.3740, \qquad \d_{2}(\a):= \frac{\gamma+4}{\gamma} \left(\frac{2+\a}{2}\ell\right)^{2+\g} .
\end{equation}
Eventually
\begin{align}	\label{soluzione-pcritico}
\widetilde u_{2,p} (x) & \to  Z_{\a,\g;\d_{2}(\a)}(x)& \\ \nonumber 
& \quad =  \log \frac{\frac{1}{2}(2+\a)^2(2+\g)^2\d_{2}(\a) |x|^{\frac{2+\a}{2}\g}}{\left(\d_{2}(\a)+ |x|^{\frac{(2+\a)(2+\g)}{2}}\right)^2}&  \text{ in } C^1_{\mathrm{loc}}(\R^2\setminus\{0\}) .
\end{align}
\end{theorem}
 This Theorem extends already known results for the Lane Emden equation to the H\'enon problem. Surprisingly the limit values of the maximum and the minimum of the radial solutions $u_p^i$ are not affected at all by the term $|x|^\a$ and they are exactly the same of the Lane Emden case which have been characterized in \cite{AdG} and \cite{GGP14}. The dependence on the parameter $\a$ appears instead in the limit of the two rescaling $U_{\a,\d(\a)}$ and $Z_{\a,\gamma:\d_2(\a)}$.

\

Section \ref{sec:3} is devoted to the computation of the  Morse index of $u_p^1$ and $u_p^2$ for large values of $p$. By taking advantage of the asymptotic study in Section \ref{sec:2} and on the characterization of the Morse index given in \cite{AG-sez2} we prove the following results
\begin{theorem}\label{teo-morse-1}
	Let $\a\geq 0$ be fixed and let $u^1_p$ be a least energy radial solution to \eqref{H} corresponding to $\a$. Then there exists $p^\star=p^\star(\a)>1$ such that for any $p>p^\star$ we have
	\begin{equation}\label{eq:morse-index-1} 
	m(u_p^1)  = 1+  2\left\lceil\frac{\alpha}{2}\right\rceil
	\end{equation}
\end{theorem}

\begin{theorem}\label{teo-morse-2}
	Let $\a\geq 0$ be fixed and  let $u^2_p$ be a least energy nodal radial solution to \eqref{H} corresponding to $\a$. Starting from the number $\ell$ determined by \eqref{ell-alpha} we set 
	\begin{align}\label{kappa}
	\kappa &= \sqrt{\frac{2+\ell^2}{2}} = \frac{2+\gamma}{2} \approx 5.1869 .
	\end{align}	
 For all $\alpha\geq 0$ there exists $p^\star_{2}=p^\star_{2}(\a)>1$ such that for any $p>p^\star_{2}$ we have
	\begin{equation}\label{eq:morse-index-2}
	m(u_p^2)  =2\left\lceil\frac {2+\a}2\kappa\right\rceil +  2\left\lceil\frac{\alpha}{2}\right\rceil 
	\end{equation}
when $\a\neq \a_n= 2(\frac{n}{\kappa}-1)$, while when $\a=\a_n$ it holds
\begin{equation}\label{eq:morse-index-2-an}
	(2+\a) \kappa  +  2\left\lceil\frac{\alpha}{2}\right\rceil\le m(u_p^2)\le  (2+\a) \kappa+  2\left\lceil\frac{\alpha}{2}\right\rceil+ 2
	\end{equation}
	\end{theorem}
Here $\lceil t\rceil = \min\{ k \in {\mathbb Z} \, : \, k\ge t\}$ stands for the ceiling function.
\\

In the case $\alpha=0$, Theorem \ref{teo-morse-2} gives back the  Morse index of the Lane-Emden problem computed in \cite{DIPN=2} since $ 2\lceil\kappa\rceil + 2\lceil0\rceil =12$.\\
 Formula \eqref{eq:morse-index-1} highlights a discontinuity of the solution's set of the H\'enon problem \eqref{H} corresponding to the even values of $\a$ which is typical of the nonlinear term $|x|^\a$ and has been already observed in  several papers among which we can quote \cite{PT}, \cite{GGN}, \cite{AG-N3}  as an example. In particular in \eqref{eq:morse-index-1} $1$ is the amount of the radial Morse index of $u_p^1$ while $2\left\lceil\frac{\alpha}{2}\right\rceil$ is the contribution of the non radial Morse index, due to the term $|x|^{\a}$, and comes from the asymptotic profile in \eqref{eq:convergenza-riscalata-1}.
 \\
Formula \eqref{eq:morse-index-2} instead exhibits two discontinuities, one corresponding to the even values of $\a$ and the other corresponding to the sequence $\a_n$ such that $\frac{2+\a}2\kappa$ is an integer. In order to analyze them we rewrite the Morse index  as
\[m(u_p^2)  =2+  2\left\lceil\frac{\alpha}{2}\right\rceil +2\left\lceil\frac {2+\a}2\kappa-1\right\rceil  \]
and we observe that $2$ is the radial contribution to the Morse index of $u_p^2$, while $2\left\lceil\frac{\alpha}{2}\right\rceil$ is the (nonradial) contribution of the rescaling of $u_p^2$ in the first nodal zone which has the same limit profile as $u_p^1$. The term $2\left\lceil\frac {2+\a}2\kappa-1\right\rceil$ is instead the (nonradial) addition coming from the rescaling of $u_p^2$ in the second nodal domain, and it is the major part of the Morse index.
What happens is that the behaviour in the second nodal zone, where the solution is smaller, has a greater influence due to  the effect of the singular term in \eqref{S-H-L}, and we will see in Section \ref{sec:3} that it gives rise to a multiplicity result.
\\
The existence of this sequence $\a_n$ seems a new phenomenon which is  peculiar of dimension $2$	since it does not appear in higher dimensions where each nodal region brings the same contribution (radial and nonradial) to the total Morse index, see \cite{AG-N3}. 
It suggests that the set of solutions to \eqref{H} changes in correspondence of that values of $\a_n$, and indeed the number of distinct nonradial solutions that we produce later on in Theorem \ref{teo:existence-2} increases by one unit.

\

Finally in Section \ref{sec:4} we give some existence and multiplicity results, by minimizing 
the energy functional $\mathcal E(u)$  on some suitable spaces of invariant functions.
	To this end  for  any integer $n\geq 1$ we denote by 
{ $R_{\frac {2\pi}n}$ any rotation of angle $\frac {2\pi}n$, centered at the origin, and we let $\mathcal G_{\frac{2\pi}n}$ be the subgroup of $O(2)$ generated by $R_{\frac {2\pi}n}$. Next, we denote by $H^1_{0,n}$ the subspace of $H^1_0(B)$ given by functions which are invariant by the action of $\mathcal G_{\frac{2\pi}n}$, namely
	\[H^1_{0,n}:=\{v\in H^1_0(B) :  \ v(x)=v(g(x))  \ \text{ for any }x\in B,   \text{ for any }g\in \mathcal G_{\frac{2\pi}n}\} ,
	\]}
 and we introduce the  $n$-invariant Nehari manifolds 
	\[\mathcal N_n:=\{u\in H^1_{0,n}: \int_B|\nabla u|^2=\int_B |x|^\a|u|^{p+1}\} ,\]
	and the nodal  $n$-invariant Nehari manifold
	\[ \begin{array}{rl}\mathcal N_{n, \nod}=\Big\{v\in H^1_{0,n} :&   v^+\neq 0, \ \int_B|\nabla v^+|^2=\int_B |x|^\a |v^+|^{p+1}, \\ &  v^-\neq 0 , \ \int_B|\nabla v^-|^2=\int_B |x|^\a |v^-|^{p+1}\Big\} .\end{array} \]
	Since $H^1_{0,n}$ is compactly embedded in  $L^p(B)$ for every $p>1$, by standard methods (see, for instance, \cite{BW}, \cite{BWW} or \cite{SerraTilli})  it follows that $\min_{u\in \mathcal N_n}\mathcal E(u)$ and  $\min_{u\in \mathcal N_{n,\nod}}\mathcal E(u)$ are nonnegative and  attained at two nontrivial functions, that we denote respectively by $u_{p,n}^1$ and $u^2_{p,n}$. They are weak and also classical solutions to \eqref{H};  $ u_{p,n}^1$ is  positive in $B$ and is  a  least energy solution in $H^1_{0,n}$, while $u^2_{p,n}$  changes sign and is  a least energy nodal solution in $H^1_{0,n}$. 
	Furthermore their {\it $n$-Morse index}, i.e. the number of negative eigenvalues of the linearized operator $L_u$	which have corresponding eigenfunction in $H^1_{0,n}$, is given by $m_n(u_{p,n}^1)=1$ and $m_n(u_{p,n}^2)=2$, because they are minima on manifolds of codimension 1 and 2, respectively. 
	Comparing the $n$-Morse index of $u^i_{p,n}$ with the $n$-Morse index of the radial solution $u_p^i$ and using a strict monotonicity result in the angular variable of \cite{Gladiali-19}, we are able to prove

	\begin{theorem}\label{teo:existence-1}
		Let $\a >0$ be fixed. Then, there exists an exponent $p^\star=p^\star(\a)$ such that problem \eqref{H} admits at least $\lceil \frac \a2\rceil  $ distinct positive nonradial solutions	for every $p> p^\star(\a)$.
	\end{theorem}
	The exponent $p^\star$ here is the same of Theorem \ref{teo-morse-1} and the nonradial positive solutions we found are invariant up to a rotation of an angle $2\pi/n$  for $n=1,\dots  \lceil \frac \a2\rceil  $, respectively. Let us remark explicitly that the first one is the least energy solution.\\
	Coming to nodal solutions, we shall prove that 
	\begin{theorem}\label{teo:existence-2}
		Let $\a\geq 0$ be fixed. Then, there exists an exponent $p_2^*= p_2^*(\a)$ such that problem \eqref{H} admits at least $\lceil \frac {2+\a}2 \kappa-1\rceil $ distinct nodal nonradial solutions for every $p >p_2^*(\a)$.
	\end{theorem}
	Here the number $\kappa$  is the same of Theorem \ref{teo-morse-2} and  $p_2^*=\max\{ 2, p_2^\star\}$ for $p_2^\star$ as in Theorem \ref{teo-morse-2}. The fact that $p_2^*$ has to be greater than $2$, instead of coincide with $p_2^\star$, is technical in order to distinguish  nonradial solutions and we do not believe it is necessary. The nonradial nodal solutions found are invariant up to a rotation of an angle $2\pi/n$ for $n=1,\dots  \lceil \frac {2+\a}2 \kappa-1\rceil$, respectively. Again, the first one is the least energy nodal solution.
	When $\a=0$  Theorem \ref{teo:existence-2} provides $5$ solutions, and gives back a previous multiplicity result in \cite{GI} to which these last two results are inspired.

		Nonradial solutions (both positive and sign-changing) have been produced also in \cite{EPW}, \cite{ZY}, \cite{KW} and \cite{Amadori} by different methods.
		 \cite{EPW}, \cite{ZY} rely on a finite dimensional reduction method and construct solutions (respectively positive and sign-changing) with $n$ symmetric concentration points placed along the vertex of a regular polygon. We mention also  \cite{EMP}, dealing with the Lane-Emden problem.
		 The symmetries of  the $n$-invariant least energy solutions are consistent with theirs, and it is reasonable to conjecture that our positive solutions $u_{p,n}^1$ coincide with the ones in \cite{EPW},  supported by the fact that we obtain the same number of solutions, but possibly this is not true anymore for nodal solutions. 
		 Indeed in the Lane-Emden case it is known that the zero set of solutions produced in \cite{EMP} touches the boundary, while \cite{GI} showed that  this does not happen to the solutions of type $u^2_{p,n}$, 
		  at least when $n=4,5$, and a similar result holds also when $\a>0$.
		 \cite{KW} and \cite{Amadori}, instead, prove a nonradial bifurcation respectively w.r.t.~the parameter $\a$, which arises in correspondence of even values of $\a$, and w.r.t. the parameter $p$.
		 Let us stress that the bifurcation in \cite{KW} allows to produce, for any given $p$, an infinite number of nonradial solutions  arising as $\a$ increases.
In a complementary way  the multiplicity results stated in Theorems \ref{teo:existence-1} and \ref{teo:existence-2} yield a finite number of solutions arising for any given value of $\a$ (imposing that $p$ is large). 
Some of such nodal solutions $u^2_{p,n}$ are nonradial for every values of $p>1$. This is certainly the case for $n=1$ (i.e. the least energy solution), and we conjecture the same holds until $n<\frac{2+\a}{2}\beta$, where $\beta\approx 2,\!305$ is a fixed number introduced in \cite{Amadori} and related to the computation of the Morse index of $u^2_p$ when $p$ approaches $1$. 
Conversely for $n=\left\lceil\frac{2+\a}{2}\beta\right\rceil, \dots \left\lceil \frac{2+\a}{2}\kappa-1\right\rceil$ the curve  $p\mapsto  u_{p,n}^2$ would coincide with the one of radial solutions $p\mapsto u^2_{p}$ for $p$ under a certain value $p_n$, and then it would bifurcate giving rise to the branch of nonradial solutions exhibited in \cite{Amadori}.

\section{Connections with the Lane-Emden problem and asymptotic profile}\label{sec:2}       
In order to study  radial solutions to \eqref{H} we let $r=|x|$ for $x\in B$ and we perform the following transformation
\begin{equation}\label{transformation-henon-no-c}
v(t)= \left(\frac{2}{2+\a}\right)^{\frac{2}{p-1}} u(r) , \qquad t= r^{\frac{2+\a}{2}},
\end{equation}
which has been introduced in \cite{GGN} and \cite{GGN2} in order to study the H\'enon problem, and transforms radial solutions to \eqref{H} into solutions of the one dimensional problem
\begin{equation}\label{LE-radial}
\begin{cases}
- \left(t v^{\prime}\right)^{\prime}= t |v|^{p-1} v  , \qquad  & 0< t< 1, \\
v'(0)=0 \ , \ v(1) =0  &
\end{cases}\end{equation}
Solutions to \eqref{LE-radial} can be seen as radial solutions to \eqref{LE} corresponding to the same exponent $p$
and the correspondence among radial solutions to \eqref{LE} and radial solutions to \eqref{H} is one-to-one. The condition $v'(0)=0$ can be not so evident and indeed it has been proved in \cite[Lemma 5.2]{AG-sez2} that any solution to \eqref{LE-radial} that satisfies
\begin{equation}\label{eq:sommabilita-v}
\int_0^1 t(v')^2\ dt<\infty\end{equation}
is a classical solution and satisfies $v'(0)=0$. 
It is then possible to apply a uniqueness result of \cite{NN} to have that for any integer $m\ge 1$ there exists only a couple of radial solutions to \eqref{LE-radial}
that are one the opposite of the other and classical solutions (see, for instance, \cite[Proposition 5.14]{AG-sez2}) which have exactly $m$ nodal zones, meaning that $u^1_p$ and $u_p^2$ are unique up to a sign.
So we denote hereafter by $v_p^1$ the unique positive solution to \eqref{LE-radial} and by $v_p^2$ any solution to \eqref{LE-radial} with $2$ nodal zones.  These solutions can be found minimizing the energy functional associated with \eqref{LE}
\[\mathcal E(w):=\frac 12 \int_B |\nabla w|^2-\frac 1{p+1}\int_B |w|^{p+1}\]
 on the radial Nehari set or on the nodal radial Nehari set, namely 
 \[\begin{split}
 & \mathcal N_{\rad}:=\{w\in H^1_{0,\rad}(B) \, : \,  \int_B|\nabla w|^2=\int_B |w|^{p+1}\} \\
  &\mathcal N_{\rad , \nod}:=\{w\in H^1_{0,\rad}(B) \, :\,  w^+\neq 0, w^-\neq 0 , w^+,w^-\in\mathcal N_{\rad} \}
  \end{split}
 \]
 and they are known as radial least energy and nodal least energy solutions to \eqref{LE}.

The asymptotic behaviour of the radial least energy solution $v_p^1$ has been studied in \cite{AdG} while the case of the radial least energy nodal solutions $v_p^2$ has been faced in  \cite{GGP14}.
Indeed radial solutions to \eqref{LE} tend to concentrate in the origin as $p$ goes to $\infty$ but, differently to what happen in the high dimensional case, the extremal values remain bounded when $p\to \infty$ so that the solutions behave like a spike and 
the concentration is different when it takes place in the first nodal domain or in the subsequent one.

The limit problem related to the first nodal domain  and hence to the positive solution $v_p^1$ is the Liouville equation  
\begin{equation}\label{Liouville}
\left\{\begin{array}{l}
-\Delta V= e^V\qquad \ \text{ in } \R^2, \\
\int_{\R^2}e^V\, dx <\infty ,
\end{array} \right.
\end{equation} 
whose radial solutions are 
\begin{align}
\label{Liouville1}V_{\d}(x) & = \log\dfrac{8\d }{(\d+|x|^{2})^2} & \text{ as }  \d >0.
\end{align}
In particular the unique solution to \eqref{Liouville} which satisfies the additional conditions $V(0)=0$  is the one with $\d=8$, i.e.
\begin{equation}\label{V}
V(x):=\log \frac{64}{\left(8+|x|^2\right)^2}.
\end{equation}
To describe the behaviour in the second nodal domain it  is also needed a singular Liouville equation, which is described in details in \cite{GGP14}, and we write here in the form
\begin{align}\label{S-L}\left\{\begin{array}{l}
-\Delta V= e^V- 2\pi \gamma \d_0 \qquad \ \text{ in } \R^2, \\
\int_{\R^2}e^V\, dx <\infty ,
\end{array} \right.
\end{align}	
where $\d_0$ stands for the Dirac measure centered at $x=0$ and $\gamma$ is a real parameter.
The family of its radial solutions is given by 
\begin{equation}
\label{Z}
Z_{\gamma;\d}(x):=\log \left( \frac{2(2+\gamma)^2\delta  |x|^{\gamma}}{    \left(\delta +|x|^{2+\gamma}\right)^2}\right)=U_{\gamma;\d}(x)+\gamma \log |x| \quad \text{ as } \d>0,
\end{equation}
where $U_{\gamma;\d}$ is a radial solution to the weighted Liouville equation \eqref{H-L} with $\a$ replaced by $\g$.
Imposing that, for some fixed $\ell>0$  that we make clear very soon, $V(t)=V(|x|)$ satisfies also 
\begin{equation}\label{eq:conditions-U}\nonumber
V(\ell)=0 \ \text{ and } \ V'(\ell)=0
\end{equation}
selects uniquely the parameters $\g$ and $\d$ as 
\begin{equation}\label{eq:relazioni-gamma-delta-l}
\gamma=\gamma(\ell)=\sqrt{2\ell^2+4} -2  \ \text{ and } \ \delta=\delta(\ell)=\frac{\gamma+4}{\gamma}\ell^{2+\gamma}.
\end{equation}
In the following we shall write $Z_\ell=Z_{\g(\ell);\d(\ell)}$ for such function.
Notice that the parameter $H$ in the notation used in \cite{GGP14} is identified by the relation
\begin{align*}
-H(\ell) & :=\int_0^{\ell} t e^{Z_{\ell}} dt  =  2(2+\gamma)\delta \int_0^{\ell} \frac{(2+\gamma) t^{1+\gamma}}{   \left(\delta+t^{2+\gamma}\right)^2} dt 
\\
& =  \frac{2(2+\gamma) \ell^{2+\g}}{ \delta+\ell^{2+\g}}  
\underset{\eqref{eq:relazioni-gamma-delta-l}}{=} \g(\ell)
\end{align*}	

Before entering the details of the asymptotic behaviour, let us spend some words about the relation between the limit problems for the Lane-Emden equation, \eqref{Liouville} and \eqref{S-L}, and the ones for the H\'enon equation, \eqref{H-L} and \eqref{S-H-L}.
	\begin{remark}\label{relazioni-limite}
		A slight variation on the transformation \eqref{transformation-henon-no-c}, namely 
		\begin{equation}\label{transformation-Liouville}
		s= \frac{2}{2+\a} r^{\frac{2+\a}{2}} , \qquad  V(s) = U(r) ,
		\end{equation}
		maps weak  radial  solutions to \eqref{Liouville} (respectively \eqref{S-L}) into weak  radial  solutions to \eqref{H-L} (respectively \eqref{S-H-L}).
		Indeed for any test function $\phi\in C^{\infty}_{0,\rad}(\R^2)$ we have
		\begin{align*}
		\int_0^{\infty} s  V' \, \phi' \,  ds - \int_0^{\infty} s e^V  \phi \, ds
		\underset{\eqref{transformation-Liouville}}{=}	\frac{2}{2+\a} \left[\int_0^{\infty} r  U' \, \psi' \,  dr - \int_0^{\infty} r^{1+\a} e^U \psi \, dr \right]
		\end{align*}
		for $\psi(r)=\phi(s)$.
		So the family of solutions of  the weighted Liouville problem \eqref{H-L} defined by \eqref{eq:U-alpha} and the one of the Liouville problem \eqref{Liouville} defined by \eqref{Liouville1} are related by 
		\begin{equation}\label{relazione:U-V} U_{\a;\left(\frac{2+\a}2\right)^2\d}(r) = V_{\d}(s) ,
		\end{equation}
		and in particular the solutions which are null at the origin are $U_{\a;\delta(\a)}(r)=V(s)$ as defined in \eqref{eq:delta-1} and \eqref{V}, respectively.
		\\
		Similarly the solutions of  singular weighted Liouville problem \eqref{S-H-L} defined by \eqref{Z-a-g}  and the ones of the singular Liouville problem \eqref{S-L} defined by \eqref{Z}  satisfy 
		\begin{equation}\label{relazione:pippo-Z}
		Z_{\a,\g; \left(\frac{2+\a}2\right)^{2+\g}\d}(r)=Z_{\gamma;\d}(s).
		\end{equation}
		In particular the additional conditions \eqref{eq:conditions-U} for \eqref{S-L} correspond to the following additional conditions for \eqref{S-H-L}
		\begin{equation}\label{eq:condition-U-H}
		U(\ell_{\a})=0, \quad U'(\ell_{\a})= 0 , \quad \text{ for } \ \ell_{\a} = \left(\frac{2+\a}{2}\ell\right)^{\frac{2}{2+\a}}
		\end{equation}
		and select uniquely the parameter $\g=\gamma(\ell)$ and $\d_2(\a)= \left(\frac{2+\a}{2}\right)^{2+\g}\d(\ell)$, where $\gamma(\ell)$ and $\d(\ell)$ are given by \eqref{eq:relazioni-gamma-delta-l}.
		It is also worth of noticing that they are the same values of the parameters selected in  \eqref{g-d-ell}, and that for this particular choice we have 
		\begin{align*}
		\int_0^{\ell_{\a}} r^{1+\a} e^{Z_{\a,\g;\delta_2(\a)}} dr \underset{\eqref{transformation-Liouville}}{=} \frac{2+\a}{2} \int_0^{\ell} s e^{Z_{\ell}(s)} ds = \frac{2+\a}{2}\g .
		\end{align*}
\end{remark}	


Some more notations are needed to describe the asymptotic behaviour of the solutions.
Concerning the positive least energy radial solution $v^1_p$, its maximum is $
v^1_p(0)$, so we introduce the scaling parameter
\begin{equation}\label{eq:epsilon-1}
\e_p:=(p
(\left(v^1_p(0)\right)^{p-1})^{-\frac 12}
\end{equation}
and the rescaled function
\begin{equation}\label{eq:v-tilde-1}
\widetilde v_p(t):=\frac{p(v^1_p(\e_p t)-v^1_p(0))}{v^1_p(0)} \ \ \text{ for } 0\le t <  \frac{1}{\e_p} . 
\end{equation}
 Extending some previous results in \cite{RW}, in \cite{AdG} it has been proved that 
\begin{proposition}\label{prop-3-1}
	Let  $v^1_p>0$ be the  radial least energy solution to \eqref{LE} related to the exponent $p$. Then as $p\to \infty$ we have $v_p^1(0)\to\sqrt e$, $\e_p\to 0$ and
	\begin{equation}\label{limit-z+}
	\widetilde v_p\to V \ \  \text{ in }C^1_{\loc}[0,\infty) .
	\end{equation}
\end{proposition}

For what concerns the least energy nodal radial solution $v^2_p$, we let 
$0<t_{1,p}< t_{2,p}=1$ its zeros, and $t_{0,p}=0$, so that its nodal zones are $B_{1,p}=[t_{0,p},t_{1,p})$ and $B_{2,p}=(t_{1,p},t_{2,p})$.  It has only two critical points, $s_{1,p}=0 \in B_{1,p}$ and $s_{2,p}\in B_{2,p}$, corresponding to two extremal values. We 
define two scaling parameters
\begin{equation}\label{eq:epsilon-2}
\e_{p}^i:=(p | v^2_p(s_{i,p} )|^{p-1})^{-\frac 12} \quad \text{ as } i=1,2
\end{equation}
and two rescaled functions 
\begin{equation}\label{eq:v-tilde-2}
\begin{split}
\widetilde v_{i,p}(t):= & \frac{p(v^2_p(\e_{p}^i t)-v^2_p(s_{i,p}))}{v^2_p(s_{i,p}) } \quad \text{ for }  0\leq t<\frac 1 {\e_{p}^i} .
\end{split}\end{equation}
The asymptotic profile of this solution $v_p^2$ has been described in the paper \cite{GGP14} where the parameters $\bar t$ and $\ell$ have been characterized. We report here a slight modified statement of their result, already appearing in \cite{DIPN=2}.
\begin{proposition}\label{prop-3-2}
	Let  $v_p^2$ be  a least energy nodal radial nodal solution to \eqref{LE} related to the exponent $p$.
Then as $p\to \infty$ we have 
\begin{equation} \label{eq:massimi-pcritico-v}
|v^2_p(0)| \to \frac {\sqrt e}{\bar t}e^{\frac {\bar t}{2(\bar t +\sqrt e)}} \approx 2.46, \quad
 |v^2_p(s_{2,p})| \to e^{\frac {\bar t}{2(\bar t +\sqrt e)}} \approx 1.17, \quad  \e_{p}^i\to 0  \  \text{ as } i=1,2
\end{equation}
where $\bar t\approx 0.7875$ is the unique root of the equation $2\sqrt e \log t+t=0$, and 
\begin{equation}
\label{limit-z+2}
\widetilde v_{1,p}\to V  \qquad \quad   \text{ in }C^1_{\loc}[0,\infty)  . 
\end{equation}
Moreover 
\begin{align}
\label{limit-zone}
 t_{1,p}\to 0 ,  \qquad  \frac{t_{1,p}}{\e^1_p} \to \infty , \qquad &  \frac{t_{1,p}}{\e^2_p} \to 0\\
\label{eq:def-l}
s_{2,p} \to 0 , \qquad  \phantom{ \frac{t_{1,p}}{\e^1_p} \to \infty , \qquad } & \frac{s_{2,p}}{\e_p^{ 2}}\to \ell \simeq 7.1979, 
\end{align}
\begin{equation}
\label{limit-z-}
\widetilde v_{2,p}\to  Z_\ell \qquad \quad   \text{ in } \ C^1_{\loc}(0,\infty) . 
\end{equation}
\end{proposition}

Since the solutions to the Lane-Emden equation are linked to the ones of the H\'enon equations by means of the transformation \eqref{transformation-henon-no-c},  the asymptotic behaviours of the last ones  stated by Theorem \ref{teo:asymptotic-of-u-1} and \ref{teo:asymptotic-of-u} follow easily Propositions \ref{prop-3-1} and \ref{prop-3-2}.
We report only the proof concerning the nodal solution $u^2_p$, because the other one is very similar.
\begin{proof}[Proof of Theorem \ref{teo:asymptotic-of-u}]
	By means of the transformation \eqref{transformation-henon-no-c}, it is clear that the items related to the Lane-Emden solution $v^2_p$ and the respective ones for the H\'enon problem $u^2_p$ are linked by the following relations 
		\[\begin{array}{cc}\
	r_{p} = t_{1,p}^{\frac2{2+\a}} , \quad & \quad \sigma_{p} = s_{2,p}^{\frac2{2+\a}} , \\
	{\mu}_{i,p} = \left(\frac{2+\a}2\right)^{\frac{2}{p-1}} |v_p^2(s_{i,p})|, \quad & \quad \rho_p^{ i} = \left(\frac{2}{2+\a} \e_p^ i\right)^{\frac2{2+\a}},
	\end{array}\]
	\[
	\widetilde u_{i,p}(r)=\widetilde v_{i,p}\left( \frac 2{2+\a} r^{\frac {2+\a}2}\right).
	\]
	Since $\left(\frac{2+\a}2\right)^{\frac{2}{p-1}}\to 1$ then $\lim_{p\to \infty}{\mu}_{i,p} =\lim _{p\to \infty}|v_p^2(s_{i,p})|$.
	So the claims concerning ${\mu}_{i,p}$, $\rho_p^i$, $r_{p}$, and $\sigma_p$ readily follows by the results recalled in Proposition \ref{prop-3-2}, in particular  the second item in \eqref{ell-alpha} is implied by \eqref{eq:def-l}.
	Eventually  \eqref{eq:convergenza-riscalata} and \eqref{soluzione-pcritico} follow by \eqref{limit-z+2} and \eqref{limit-z-}, by the computations made in Remark \ref{relazioni-limite}.
\end{proof}

\section{The Morse index of $u_p^1$ and $u_p^2$}\label{sec:3}

In this section we address to the computation of the Morse index of  radial least energy solutions $u_p^1$ and $u_p^2$ when $p$ goes to $\infty$.
By definition  the Morse index of a radial solution $u_p$ to \eqref{H}, that we denote by $m(u_p)$, is  the maximal dimension of a subspace of $H^1_0(B)$ in which the quadratic form 
$Q_u$ is negative defined, or equivalently, is the number, counted with multiplicity, of negative eigenvalues in $H^1_0(B)$ of
\begin{equation}\label{eigenvalue-problem}
\left\{\begin{array}{ll}
-\Delta  \phi-p|x|^\a {    |u_p|^{p-1}}  \phi =\L_h(p)\, \phi & \text{ in } B\\
\phi= 0 & \text{ on } \partial B.
\end{array} \right.
\end{equation}
Similarly the radial Morse index of $u_p$, denoted by $m_\rad(u_p)$, is  the number  of negative eigenvalues of  \eqref{eigenvalue-problem} in $H^1_{0,\rad}(B)$, namely the eigenvalues of \eqref{eigenvalue-problem} associated with a radial eigenfunction.
It is known by  \cite[Theorem 1.7]{AG-sez2} that the radial Morse index is equal to the number of nodal zones, 
 that is $m_{\rad}(u^1_p)=1$ and $m_{\rad}(u^2_p)=2$. 
In the mentioned paper \cite{AG-sez2} it has been proved that the Morse index (or radial Morse index) of $u_p$ is the number, counted with multiplicity, of negative eigenvalues $\widehat\Lambda _h(p)$ (negative radial eigenvalues $\widehat\Lambda _h^\rad (p)$ resp.)
of the singular eigenvalue problem 
\begin{equation}\label{singular-eigenvalue-problem}
\left\{\begin{array}{ll}
-\Delta  \widehat\phi-p|x|^\a|u_p|^{p-1}\widehat\phi =\dfrac{\widehat\L_h(p)}{|x|^2}\widehat\phi & \text{ in } B\setminus\{0\}\\
\widehat\phi= 0 & \text{ on } \partial B,
\end{array} \right.
\end{equation}
in $\mathcal H_0$ ($\mathcal H_{0,\rad}$ resp.). Here $\mathcal H_0:=H^1_0\cap \mathcal L$ and  $\mathcal L$ is the Lebesgue space
\[{\mathcal L}:=\{w:B\to \R \text{ measurable and s.t. } \int_B |x|^{-2} w^2 dx < +\infty\}\]
  with the scalar product
  $\int_B |x|^{-2}\psi w\ dx$, which gives the orthogonality condition
  \[ w\underline \perp \psi \Longleftrightarrow \int_B |x|^{-2} w \psi dx =0  \  \text{for }w,\psi\in \mathcal L\]
  and $\mathcal L_{\rad}$ and $\mathcal H_{0,\rad}$ are their subspaces given by radial functions.
Of course $\mathcal H_0$ ($\mathcal H_{0,\rad}$ resp.)
are Banach and Hilbert spaces with the norm 
\[ \|w\|_{\mathcal{H}_0}^2=\int _B |\nabla w|^2 + |x|^{-2}w^2\ dx.\] 
By weak solutions to \eqref{singular-eigenvalue-problem} we mean a function $\widehat \phi \in \mathcal H_0$  that satisfies
\begin{equation}\label{eq:weak}
\int_B \nabla \widehat \phi \nabla w-  p|x|^\a |u_p|^{p-1} \widehat \phi w=\widehat \L_i(p) \int_B |x|^{-2} \widehat \phi w 
\end{equation}
 for any $w\in \mathcal H_0$. Let us remark that, since $C^{\infty}_0(B\setminus\{0\})$ is dense in $\mathcal H_0$ with respect to the norm $\|\cdot
\|_{\mathcal H_0}$ (see Lemma \ref{prop-density} in the Appendix), it is enough to take the test functions $w$ in \eqref{eq:weak} in $C^{\infty}_0(B\setminus\{0\})$.
Nevertheless by \cite[Prop. 3.1]{AG-sez2} a weak solution $\widehat \phi$ is a classical solution to \eqref{singular-eigenvalue-problem} in $B\setminus\{0\}$.
\\
Moreover these singular eigenvalues $\widehat \L_h(p)$ have the useful property that can be decomposed as
\begin{equation}\label{decomposition}
\widehat \L _h(p)=\widehat \L_j^{\rad}(p)+k^2,
\end{equation}
where $k^2$ are the eigenvalues of the Laplace-Beltrami operator on ${\mathbb S}_1$, and $\widehat \L_j^{\rad}(p)$ are the radial singular eigenvalues of \eqref{singular-eigenvalue-problem},
 which are all simple by \cite[Property 5]{AG-sez2}. Then an eigenfunction $\widehat \phi_h\in \mathcal H_0$ corresponding to $\widehat \L_h(p)$ is given, in polar coordinate $(r,\theta)$, by 
\begin{equation}\label{eq:autofunzioni}
\widehat \phi_h(r,\theta)=\widehat \phi_j^\rad (r)(A \cos (k \theta)+B\sin (k \theta))
\end{equation}
where $\widehat \phi_j^\rad \in \mathcal H_{0,\rad}$ is an eigenfunction associated with $\widehat \L_j^{\rad}(p)$ and $A,B\in \R$. \\
This decomposition allows from one side to easily compute the Morse index of a radial solution $u_p$  knowing only the radial eigenvalues $\widehat \L_j^{\rad}(p)$ and, from the other side, is useful to understand the feasible symmetries that nonradial solutions can have in order to prove the existence results, see Section \ref{sec:4}.
\\
Performing again the transformation in \eqref{transformation-henon-no-c} and letting $\psi(t)=\widehat \phi(r)$ we have that the computation of the Morse index is linked to the singular Sturm-Liouville problem 
	\begin{equation}\label{radial-general-H-no-c}
	\left\{\begin{array}{ll}
	- \left(t \psi'\right)'- t p|v^i_p|^{p-1} \psi = t^{-1} {\widehat \nu}_j(p)  \psi & \text{ for } t\in(0,1)\\
	\psi\in  \mathcal H_{0,\rad}
	\end{array} \right.
	\end{equation} 
	where $v^i_p$ for $i=1,2$ is defined in \eqref{transformation-henon-no-c}
	and the radial singular eigenvalues $\widehat \L_j^\rad(p)$ are linked to the singular eigenvalues $\widehat \nu_j(p)$ by the relation
	\[\widehat \L_j^\rad(p)=\frac {(2+\a)^2}{4}\widehat \nu_j(p)\]
	see  \cite[Lemma 5.7]{AG-sez2}.
	Recall that $\psi$ is a weak solution to \eqref{radial-general-H-no-c} means that 
	\begin{equation}\label{radial-general-weak-H-no-c}
	\int_0^1 t \psi'\varphi' \ dt- p\int_0^1 t |v^i_p|^{p-1} \psi\varphi\ dt={\widehat \nu}_j(p)\int_0^1  t^{-1}   \psi\varphi \ dt
	\end{equation}
	for every $\varphi \in \mathcal H_{0,\rad}$  or for every $\varphi \in C^{\infty}_{0,\rad}(B\setminus\{0\})$.\\
The space $\mathcal H_{0,\rad}$ 
	 is introduced to obtain compactness in the variational formulation of \eqref{singular-eigenvalue-problem}  or, equivalently, \eqref{radial-general-H-no-c}. Anyway compactness is possible only for negative eigenvalues, as rigorously proved in  \cite{GGN2}. 
	As far as 	
	\[ {\mathcal R}^i (\phi): = \frac{\int_0^1 t\left(|\phi'|^2 -p|v^i_p|^{p-1}\phi^2 dt\right) dr }{\int_0^1 t^{-1}\phi^2 dt}
	\]
	has a negative infimum on $\mathcal H_{0,\rad}$, such infimum is attained by a function $\psi_{1,p}\in \mathcal H_{0,\rad}$ which is a weak solution to \eqref{radial-general-H-no-c} corresponding to
	\begin{equation}\label{nu-var-1}
	{\widehat \nu}^i_1(p)=\min \left\{ {\mathcal R}^i (\phi) \, : \, \phi\in {\mathcal H}_{0,\rad} , \; \phi\neq 0 \right\} .
	\end{equation}
	Next if ${\mathcal R}^i$ has a negative infimum also in the subspace of $\mathcal H_{0,\rad}$ orthogonal to $\psi_{1,p}$, meaning that
	\begin{equation}\label{orto}
	\phi \underline{\perp}  \psi  \Longleftrightarrow \int_0^1 t^{-1} \phi \psi dt =0 ,\end{equation}
	such infimum is attained by a function $\psi_{2,p} \in \mathcal H_{0,\rad}$, $\psi_{2,p}\underline{\perp}  \psi_{1,p}$, which is a weak solution to \eqref{radial-general-H-no-c} corresponding to
	\begin{equation}\label{nu-var-i}
	{\widehat \nu}^i_{2}(p)=\min \left\{ {\mathcal R}^i (\phi) \, : \, \phi\in   {\mathcal H}_{0,\rad} , \; \phi\neq 0 , \; \phi\underline \perp \psi_{1,p} \right\},
	\end{equation}
	and the procedure can be iterated.
	\\
	These generalized radial singular eigenvalues ${\widehat \nu}^i_j(p)$, (associated with $v^1_p$ or $v^2_p$) have been studied in \cite[Subsections 3.1 and 5.3]{AG-sez2} where it is proved, among other things,  that they are all {\em simple} and the only negative eigenvalues of \eqref{radial-general-H-no-c} are 
	\begin{align}
	\label{stima-primo-autov} &  -1 < \widehat \nu_1^1(p)<0 , &   \text{ and }\\
	\label{stima-autov} & \widehat \nu_1^2(p)<-1<  \widehat \nu_2^2(p)<0 & 
	\end{align}
	for any value of the parameter $p$. 
	Besides  the radial Morse index of $u_p^1$ and $u_p^2$ coincides with the number of negative eigenvalues of \eqref{radial-general-H-no-c} (see \cite[Lemma 5.7 and Remark 5.12]{AG-sez2}). 
	Furthermore also the Morse index can be computed starting from the singular eigenvalues as follows: 
	\begin{proposition}\label{general-morse-formula-H}
			 For every $\a\ge 0$ the Morse index of $u^i_p$ is given by
		\begin{align}\label{tag-2-H}
		m(u^i_p) & = 2 \sum\limits_{j=1}^{i} \left\lceil \frac{2+\a}{2} \sqrt{-  \widehat {\nu}^i_j(p)} \right\rceil    -i 
		\end{align}
	 as $i=1,2$.
\end{proposition}
	Indeed formula \eqref{tag-2-H} is obtained putting together Proposition 1.5,  Theorem 1.7 and Remark 5.12 from \cite{AG-sez2}, and recalling that in dimension $N=2$ the multiplicity of the eigenvalues $\lambda_j=-j^2$ of the Laplace-Beltrami operator are $N_0=1$, $N_j=2$ for $j\ge 1$.

	Therefore the Morse index of least energy radial solutions for large values of $p$ can be deduced by the asymptotic behaviour of the singular eigenvalues and of the related eigenfunctions.
	It is therefore needed to look at the limit eigenvalue problem, which can be deduced from \eqref{radial-general-H-no-c} via the asymptotic behaviour of the functions $v^i_p$ recalled in previous Section. As the latter depends heavily by the number of nodal zones, we deal first with the minimal energy radial solution $u^1_p$ (in Subsection \ref{sec:3.1}), and then with the minimal energy nodal radial solution $u^2_p$ (in Subsection \ref{sec:3.2}).

\subsection{The case of the positive solution $u_p^1$}\label{sec:3.1}
In this subsection we analyze the least energy radial solution $u_p^1$ in order to compute its Morse index, depending on $\a$ when $p$ is sufficiently large. By the aforementioned results in this case \eqref{radial-general-H-no-c} has only one negative eigenvalue which will be simply denoted by $\nu_1(p)$ henceforth. 
It satisfies  \eqref{stima-primo-autov} and formula \eqref{tag-2-H} simplifies into
\begin{align}\label{tag-2-H-1}
m(u^1_p) & = 2 \left\lceil \frac{2+\a}{2} \sqrt{- {\nu}_i(p) }\right\rceil    - 1 .
\end{align}
Then the result in Theorem \ref{teo-morse-1} is a consequence of Proposition \ref{general-morse-formula-H} once we have proved 
that:
\begin{proposition}\label{prop:primo-autovalore-1}
	Let $\nu_1(p)$ the unique radial singular negative eigenvalue of \eqref{radial-general-H-no-c}  corresponding to $v^1_p$. Then 
	\begin{align}\label{lim-autov-palla-1-sol-positive}
	&\lim_{p\to \infty}  \nu_1(p)=-1
	\end{align}
\end{proposition}

Before proving Proposition \ref{prop:primo-autovalore-1}, which is the core of the present subsection, let us deduce Theorem \ref{teo-morse-1} from it.

{ \begin{proof}[Proof of Theorem \ref{teo-morse-1}]
		The limit \eqref{lim-autov-palla-1-sol-positive}, together with \eqref{stima-primo-autov}, ensures that $ \nu_1(p) \to -1$ from above, so that $	\frac{2+\a}{2} \sqrt{- {\nu}_1(p) } \to   \frac{2+\a}{2} $ from below and then \eqref{tag-2-H-1} gives
		\[ m(u^1_p) \to  2 \left\lceil \frac{2+\a}{2} \right\rceil    - 1  = 1 + 2\left\lceil \frac{\a}{2} \right\rceil  \]
		because the ceiling function is left-continuous. 
		The conclusion follows because the Morse index is a discrete quantity, and therefore it must be definitely equal to its limit.
	\end{proof}

	Proposition \ref{prop:primo-autovalore-1}, in turn, is proved by sending $p\to \infty$ in the Sturm-Liouville problem \eqref{radial-general-H-no-c}, or better into its  rescaled version.
	To enter the details we define
	\begin{align}\label{rescaled-eigenf-1}
	\widetilde \psi_{p}(r):= \begin{cases} \psi_{p}\left(\e_p r\right) &  \text{ for } r\in [0,1/\e_p) \\
	0 & \text{elsewhere.}\end{cases}
	\end{align}
	 for $\e_p$ as in \eqref{eq:epsilon-1}.
Since  $\psi_p\in {\mathcal H}_{0,\rad}$ then $\widetilde \psi_{p} \in \mathcal D_{\rad}$, where 
\begin{equation}\label{def:D}
\mathcal D_{\rad} = D^{1,2}(\R^2) \cap \mathcal L_{\rad}(\R^2).
\end{equation}
	Here  $D^{1,2}(\R^2)$, as usual,  is the closure of $C^{\infty}_0(\R^2)$ under the $L^2$-norm of the gradient, and 
	$\mathcal L_{\rad}(\R^2)$ denotes 
	the space of measurable functions $g: [0,\infty)\to \R$ such that $\int_0^\infty t^{-1}g^2 \ dt<\infty$. 
 $\mathcal D_{\rad} $ has the Hilbert and Banach structure induced by the norm
\[\| \psi\|_{\mathcal D_{\rad}} = \left(\int_0^{\infty} \left(r^{-1}\psi^2 + r|\psi'|^2\right) dr \right)^{\frac{1}{2}}, \]
 and it actually coincides with the closure of $C^{\infty}_{0}(0,\infty)$ under this norm (see  Lemma \ref{lem-Drad} in the Appendix). 

Moreover  $\psi_p$ is an eigenfunction for \eqref{radial-general-H-no-c} related to the eigenvalue $\nu_1(p)<0$ if and only if $\widetilde \psi_{p}$ satisfies 
	\begin{equation}\label{radial-general-H-no-c-resc-1}
	- \left( r\widetilde \psi'\right)'-   r W_p \widetilde\psi = r^{-1} {\nu}_1(p)  \widetilde\psi   \qquad  \text{ as } r\in (0,1/\e_p) \\ 
	\end{equation} 
	Here 
	\begin{equation}\label{W-1}
	W_p(r)= 
	p \e_p^2 \left|v^1_p(\e_p r)\right|^{p-1} = \left|1+\frac{1}{p} \widetilde v_p(r)\right|^{p-1} \qquad   \text{ for } r\in [0,1/\e_p) . 	
	\end{equation} 	
	By the convergence result in \cite{GGP14} recalled in \eqref{limit-z+} as $p\to\infty$ we have 
	\begin{equation}\label{limit-W-1}
	\begin{split}
	W_p(r)\to W^1(r)=e^{V(r)} = \frac{64}{(8+r^2)^2}  & \text{ in } C^0_{\loc}[0,\infty)
	\end{split}\end{equation}
	where  $V$ has been defined in \eqref{V}. Therefore the natural limit problem for \eqref{radial-general-H-no-c-resc-1} is 
	\begin{align}\label{eq:finale-1}
	\begin{cases} -\left(r \phi'\right)'= r \left(W^1+\frac{\beta }{r^2} \right) \phi \quad &   r> 0, \\
\phi \in {\mathcal D}_{\rad} &  
	\end{cases}
	\end{align} 	
	As usual we mean that $\phi$ solves the equation in weak sense, i.e.
		\begin{align}\label{eq:finale-sol}
		\int_0 ^{\infty} r \phi' \varphi' \, dr = \int_0^{\infty} r \left( W^1 +\frac{\beta }{r^2} \right) \phi \varphi \, dr  
				\end{align}
	 	for every $\varphi \in {\mathcal D}_{\rad}$ or equivalently for every $\varphi \in C^{\infty}_0(0,\infty)$.
	 	\\
	Eigenvalues to \eqref{eq:finale-1} are attained in $\mathcal D_\rad$ as far as they are negative. In particular the first  eigenvalue 
	of \eqref{eq:finale-1} is $\beta_1=-1$, which is simple and attained by the function 
		\begin{equation}\label{prima-autof-limite-1}
		\eta_1(r):=\frac{4r}{8+r^2} ,
		\end{equation}
	as shown in \cite[Sec 5]{DIPN=2}. It is 
	the unique negative eigenvalue, see Proposition \ref{prop:beta_1} in the Appendix.

	We are now ready to prove Proposition \ref{prop:primo-autovalore-1}.

\begin{proof}[Proof of Proposition \ref{prop:primo-autovalore-1}]

		Since $\nu_1(p)>-1$ for every $p$ by \eqref{stima-primo-autov}, it suffices to show that $\limsup\limits_{p\to \infty}\nu_1(p) \le  - 1$.  
				We shall prove it by a suitable  choice of the test functions in  the variational characterization  \eqref{nu-var-1} showing that, for every $0<\e<1$ there exists an exponent $p_\e>0$ such that $\nu_1(p)\leq -1+\e$ for $p\geq p_\e$.\\
		Let us take a cut-off function $\Phi\in C^{\infty}_0(0,\infty)$ such that  
		\begin{equation}\label{eq:cut-off}
		 0\le \Phi(r) \le 1 , \quad \left| \Phi'(r)\right|\leq  \frac 2{R}, \quad 
		\Phi(r)=\begin{cases}
		1 & 0\leq r< R , \\
		0 & r>\ {2R} .
		\end{cases}  \end{equation}
		Letting $\e_p$ and $\eta_1$ as defined in \eqref{eq:epsilon-1} and  \eqref{prima-autof-limite-1}, respectively, we set 	
		\begin{align}
		\label{psi-epsilon}
		\varphi_p(r)=\eta_1\left(\frac{r}{\e_p}\right) \Phi \left(\frac r{ \e_p} \right)  , \quad \text{ as } r\in[0,1] .
		\end{align}
			The function $\eta_1$  is decreasing on $\left(2^{\frac{3}{2}}, \infty\right)$ with  $\lim\limits_{r\to \infty}\eta_1(r)=0$, and $\int_0^{\infty} r^{-1} \eta^2_1 dr =1$. So we can choose $R=R(\e)$ in such a way that
			\begin{align}\label{R-e-1}
			\eta_1(r)\le \eta_1(R)<\frac{\e}{4} \quad \text{ for  $r>R$, }\\
			\label{R-e-2}
			\int_0^{\infty} r^{-1} \eta^2_1 \Phi^2 dr \ge\int_0^{R} r^{-1} \eta 2_1  dr  \ge 1-\e /2. 
			\end{align}
		\\
		Notice that since $\e_p\to 0$ we may assume w.l.g.~that $p$ is so large that $1/\e_p > 2R$, so that $\varphi_p\in \mathcal H_{0,\rad}$.
		Inserting the test function $\varphi_p$ into the variational characterization \eqref{nu-var-1} gives
		\begin{align*}
		\nu_1(p) & \le \dfrac{\int_0^1 r\left(|\varphi_p'|^2 -p|v_p|^{p-1}\varphi_p^2 dr\right) dr}{\int_0^1  r^{-1}\varphi^2 dr } ,
		\end{align*}
		next we compute all the terms. \\
		First, since  for every functions $f$ and $g$ we have $\left[\left(fg\right)'\right]^2 = f'\left(fg^2\right)'+ f^2 (g')^2$
		we get
		\begin{align*}\nonumber 
		\int_0^1r |\varphi_p'|^2 dr & =\int_0^1r\left(\eta_1\left(\frac{r}{\e_p}\right)\right)' \left(\eta_1\left(\frac{r}{\e_p}\right) \, \Phi_p^2\left(\frac r{ \e_p} \right)\right)' dr \\ \nonumber
		& \quad + \int_0^1r {\eta_1}^2\left(\frac{r}{\e_p}\right) \left[\left(\Phi_p\left(\frac r{ \e_p} \right)\right)'\right]^2 dr
		\intertext{and rescaling}
		& = \int_0^{\frac1{\e_p}} s\eta_1'(s) \left(\eta_1(s)\Phi^2( s)\right)'ds 
		+ \int_0^{\frac1{\e_p}} s \eta_1^2(s) \left[\left(\Phi (s)\right)'\right]^2 ds
		\end{align*}
		Concerning  the first integral, because $\Phi$ has compact support contained in $[0,2R]$ we have
		\begin{align*} \nonumber
		\int_0^{\frac1{\e_p}} s
		\eta_1'\left(\eta_1\Phi^2\right)' ds
		= \int_0^{+\infty} s \eta_1' \left(\eta_1\Phi^2\right)' ds
		\intertext{and since  $\eta_1$ solves \eqref{eq:finale-1} corresponding to  $\beta_1=-1$ we get} 
		= -\int_0^{+\infty} s^{-1}\eta_1^2\Phi^2 ds + \int_0^{+\infty} s  W^1 \eta_1^2\Phi^2 ds .
		\end{align*}
		Therefore
		\begin{align}	\label{f4}
		\int_0^1r |\varphi_p'|^2 dr & =-\int_0^{+\infty} s^{-1}\eta_1^2\Phi^2 ds + \int_0^{+\infty} s  W^1 \eta_1^2\Phi^2 ds+ \int_0^{\infty} s \eta_1^2 (\Phi')^2 ds
		\end{align}
		
		Next we compute 
		\begin{align}\nonumber
		\int_0^1r p| v_p^1|^{p-1}\varphi_p^2 dr& = \int_0^1 r p|  v_p^1|^{p-1}\left(\eta_1\left(\frac{r}{\e_p}\right) \Phi_p\left(\frac r{ \e_p} \right) \right)^2 dr
		\intertext{rescaling and using the properties of $\Phi$ we get}
		\label{step3} & = \int_0^{\infty} s W_p\eta_1^2\Phi^2 ds  
		\end{align}
		where $W_p$ has been introduced in \eqref{W-1}.
		Similarly 
		\begin{align}\nonumber
		\int_0^1 r^{-1}\varphi_p^2dr &=\int_0^1 r^{-1} \left(\eta_1\left(\frac{r}{\e_p}\right) \Phi_p\left(\frac r{ \e_p} \right) \right)^2 dr \\
		\label{step4} & = \int_0^{\infty} s^{-1} \eta_1^2\Phi^2(s) ds 
		\end{align}
		
		Putting \eqref{f4}, \eqref{step3} and \eqref{step4} into the variational characterization \eqref{nu-var-1} gives
			\begin{align*}
			\nu_1(p) & \le \dfrac{\int_0^1 r\left(|\varphi_p'|^2 -p|v^1_p|^{p-1}\varphi_p^2 \right) dr}
		{\int_0^1  r^{-1}\varphi_p^2 dr } \\
		&=   \dfrac{-\int_0^{\infty} s^{-1} \eta_1^2  \Phi^2  ds + \int_0^{\infty} s (W^1-W_p) \eta_1^2  \Phi^2  ds + \int_0^{\infty} s \eta_1^2 \left(\Phi'\right)^2 ds }
		{\int_0^{\infty}  s^{-1}\eta_1^2 \Phi^2 ds } \\
		& = -1 +  \dfrac{\int_0^{\infty} s (W^1-W_p) \eta_1^2  \Phi^2  ds + \int_0^{\infty} s \eta_1^2 \left(\Phi'\right)^2 ds }
		{\int_0^{\infty}  s^{-1}\eta_1^2 \Phi^2 ds }.
		\end{align*}
		But using the explicit law for $\eta_1$ given in \eqref{prima-autof-limite-1} and the properties of $\Phi$ we have 
			\begin{align*} 
			\int_0^{\frac1{\e_p}} s \eta_1^2(s) \left(\left(\Phi\right)'\right)^2 ds \leq
			\frac{4}{R^2}\int_R^{2R}  s \eta_1^2(s)ds \underset{\eqref{R-e-1}}{<} \frac{\e^2 }{4 R^2}  \int_R^{2R} s\  ds = \frac{3\e^2}{8}  < \frac{3\e}{8} ,
			\end{align*}
		so that	
		\begin{align*}
		\nu_1(p) & < -1 +  \dfrac{\int_0^{\infty} s (W^1-W_p) \eta_1^2  \Phi^2 (s) ds + \frac{3\e}{8}  }
		{\int_0^{\infty}  s^{-1}\eta_1^2 \Phi^2 (s)ds} \\
		& \underset{\eqref{R-e-2}}{<} - 1 + \frac{\int_0^{+\infty} s  \left|W^1 -W_p\right| \eta_1^2\Phi^2 ds+ \frac{3}{8} \e}{1-\e /2} .
		\end{align*}
	On the other hand by the properties of $\Phi$  we have
\begin{align*}
\int_0^{+\infty} s  \left|W^1 -W_p\right| \eta_1^2\Phi^2 ds \le \sup_{(0,2R)}|W^1-W_p|\int_0^{2R}s \eta_1^2 ds 
\end{align*}
and since $W_p\to W^1$ uniformly on $[0,2R]$ we can take $p_{\e}$ in dependence by $\e$ and $R(\e)$ large enough such that 
\[ \sup_{(0,2R)}|W^1-W_p| \le \frac{\e }{8\int_0^{2R}s \eta_1^2 ds }  \quad \text{ for } p> p_{\e}.\]
Eventually we end up with $\nu_1(p) < - 1 + \frac{\e}{ 2-\e} < -1+\e$.

\end{proof}

\subsection{The case of the nodal solution $u_p^2$}\label{sec:3.2}
For the nodal solution  $u^2_p$ problem \eqref{radial-general-H-no-c} has two negative eigenvalues, that will be simply denoted by $\nu_1(p)$ and $\nu_2(p)$ in the following, and satisfy \eqref{stima-autov}.	Therefore to  compute the asymptotic Morse index  according to Proposition \ref{general-morse-formula-H} we need to compute the limit of the two negative eigenvalues  $\nu_1(p)$ and $\nu_2(p)$, and precisely we shall see that:
\begin{proposition}\label{prop-3-5}
	Let $ \nu_1(p)$  and  $\nu_2(p)$ be the radial singular negative eigenvalues of \eqref{radial-general-H-no-c} and let 
	$\kappa $ be as defined in \eqref{kappa}. Then 
	\begin{align}\label{lim-autov-palla-1}
	&\lim_{p\to \infty} \nu_1(p)=
	-\kappa^2 \simeq -26.9
	\\ \label{lim-autov-palla-2}
	&\lim_{p\to \infty} \nu_2(p)=-1	. \end{align}
\end{proposition}
Before going on, let us see how Theorem \ref{teo-morse-2} can be easily deduced by Propositions \ref{general-morse-formula-H} and \ref{prop-3-5}
	\begin{proof}[Proof of Theorem \ref{teo-morse-2}]
		Thanks to \eqref{stima-autov} and \eqref{lim-autov-palla-2} one can see as in the proof of Theorem \ref{teo-morse-1} that 
		\[\left\lceil \frac{2+\a}{2} \sqrt{- {\nu}_2(p) }\right\rceil  
		\to 1+ \left\lceil \frac{\a}{2} \right\rceil .\]
		Besides \eqref{lim-autov-palla-1} yields that $\frac{2+\a}{2} \sqrt{- {\nu}_1(p) } \to  \frac{2+\a}{2} \kappa$, and then 
		\[ \left\lceil \frac{2+\a}{2} \sqrt{- {\nu}_1(p) }\right\rceil  
		\to \left\lceil \frac{2+\a}{2}\kappa \right\rceil ,\] provided that $\frac{2+\a}{2}\kappa$ is not integer, that is $\a\neq \a_n$.
		In this case formula \eqref{tag-2-H} yields
		\[m (u^2_p) \to  2 \left\lceil \frac{\a}{2} \right\rceil + 2 \left\lceil \frac{2+\a}{2}\kappa \right\rceil    
		\]
		and the first part of the claim follows since the Morse index is a discrete quantity.
	Otherwise we cannot pass to the limit inside formula \eqref{tag-2-H}, because the ceiling function is not continuous at $\frac{2+\a_n}{2}\kappa$. Nevertheless the just exposed arguments show that 
		\[\left\lceil\frac{2+\a_n}{2}\sqrt{-\nu_1(p)}\right\rceil \in \left\{ \frac{2+\a_n}{2}\kappa, \frac{2+\a_n}{2}\kappa+1\right\}\] for large values of $p$, which concludes the proof. 
	\end{proof}

To prove Proposition \ref{prop-3-5} we begin by taking $\psi_{j,p} \in \mathcal H_{0,\rad}$, the eigenfunctions of \eqref{radial-general-H-no-c} corresponding to    $v_p^2$  and to $\nu_j(p)$ for $j=1,2$ normalized such that	
	 \begin{equation}\label{normalization}
	\int_0^1 r^{-1} \psi_{j,p}\psi_{k,p} dr=\delta_{jk}.
	\end{equation}
	Next, using the notations introduced in Section \ref{sec:2}, 
	for $j=1,2$ we define the rescaled eigenfunctions
	\begin{align}\label{rescaled-eigenf}
	\widetilde \psi_{j,p}^i(r):= \begin{cases}
	\psi_{j,p}\left(\e_{p}^i r\right) & \text{ for }   \frac {t_{i-1,p}}{\e_p^i}< r <\frac {t_{i,p}}{\e_{p}^i}\\
	0 & \text{ elsewhere ,}
	\end{cases}
	\end{align}
	with $\e_{p}^i$ as in \eqref{eq:epsilon-2}, in such a way that	
	\begin{align}
	&\int_0^\infty r^{-1}\left(\widetilde \psi_{j,p}^i\right)^2\ dr=
	\label{eq:riscalata-1}
	\int_{t_{i-1,p}}^{t_{i,p}} r^{-1}\psi_{j,p}^2\ dr  \le\int_0^1 r^{-1}\psi_{j,p}^2\ dr=1\\
	&\int_0^\infty r\left( (\widetilde \psi_{j,p}^i)'\right)^2\ dr=
	\label{eq:riscalata-2}
	\int_{t_{i-1,p}}^{t_{i,p}}r\left(\psi_{j,p}'\right)^2\ dr\le    \int_0^1 r
	\left(\psi_{j,p}'\right)^2\ dr.
	\end{align}

Then the functions $\widetilde \psi_{j,p}^i$ belong to the space $\mathcal D_{\rad}$ introduced in \eqref{def:D}
	and they satisfy
	\begin{equation}\label{eq:riscalata}
	-\left(r (\widetilde\psi_{j,p}^i)'\right)'=r\left(W_{p}^i +\frac{\nu_j(p)}{r^2}\right)\widetilde\psi_{j,p}^i \ \ \text{ as }  \frac {t_{i-1,p}}{\e_{p}^i}< r <\frac {t_{i,p}}{\e_{p}^i}
	\end{equation}
	for 
	\begin{align}\label{rescaled-potential}
	W_{p}^i(r):= & p(\e_{p}^i)^2\left|v_p^2(\e_{p}^ir)\right|^{p-1} =  \left| 1+\frac{\widetilde v_{i,p}(r)}p\right|^{p-1} .
	\end{align}
	Equation \eqref{eq:riscalata} is meant in weak sense, namely
	\begin{equation}\label{eq:riscalata-sol}
	\int_0^{\infty} r (\widetilde\psi_{j,p}^i)' \varphi'= 	\int_0^{\infty} r\left(W_{p}^i +\frac{\nu_j(p)}{r^2}\right)\widetilde\psi_{j,p}^i \varphi \, dr 
	\end{equation}
	for every $\varphi\in C^{\infty}_0(0,\infty)$ 	
	whose support is contained in $\left(\frac {t_{i-1,p}}{\e_{p}^i} , \frac {t_{i,p}}{\e_{p}^i}\right)$.\\
	Proposition \ref{prop-3-2} yields that when $p\to\infty$ 
	 \begin{align}\label{limit-W-1-2}
	 W_{p}^1(r)\to W^1(r)=e^{V(r)} = \frac{64}{(8+r^2)^2}  & \text{ in } C^0_{\loc}[0,\infty) \\
	 \label{limit-W-2-2}
	 W_{p}^2(r)\to W^2(r)=e^{Z_{\gamma;\delta}(r)} =  \frac{2(2+\gamma)^2\delta  r^{\gamma}}{    \left(\delta +r^{2+\gamma}\right)^2} & \text{ in } C^0_{\loc}(0,\infty) 
	 \end{align}
	 where  $V$ and $Z_{\gamma;\delta}$ have been defined in \eqref{V} and \eqref{Z}, respectively and $\gamma$ and $\delta$ are fixed in \eqref{eq:relazioni-gamma-delta-l}. 
	 Therefore the natural limit problems for \eqref{eq:riscalata} are 
	 \begin{align}\label{eq:finale}
	 \begin{cases} -\left(r \psi'\right)'= r \left(W^i+\frac{\beta^i }{r^2} \right) \psi \quad &   r\in(0,\infty), \\
	 \phi\in {\mathcal D}_{\rad}   &\end{cases}
	 \end{align} 	
	 whose weak solutions are meant in the sense of \eqref{eq:finale-sol}. \\
	 
For $i=1$ \eqref{eq:finale} coincides with \eqref{eq:finale-1}  and, as already recalled,  it has only one negative eigenvalue
	$\beta_{1}=-1$ with eigenfunction $	\eta_1$  given by \eqref{prima-autof-limite-1}.
	 Also for $i=2$  there  is only one negative eigenvalue 
	\[\beta^2_{1}=-\kappa^2,\] 
	where $\kappa= \sqrt{\frac{2+\ell^2}{2}}=\frac {2+\g}2$  is the fixed number introduced  in \eqref{kappa}. Such negative eigenvalue is simple and its related eigenfunction is
		\begin{equation}\label{eta-2-1}
		\eta_1^2(r):=\frac {\sqrt {2\kappa \delta} \, r^\kappa}{\delta+r^{2\kappa}} , 
		\end{equation}
	 see Proposition \ref{prop:beta_2} in the Appendix.

\

The proof of Proposition \ref{prop-3-5} is quite long and involved. We divide it in two parts by dealing first with the first eigenvalue  and after with the second one. In doing this we also compute the limits of the rescaled eigenfunctions and show that

\begin{proposition}[First part of Proposition \ref{prop-3-5}]\label{prop-3-5-1}
	Let $\nu_1(p)$ be the first radial singular negative eigenvalue of \eqref{radial-general-H-no-c} corresponding to $v_p^2$. Then as $p\to\infty$ we have 
	\begin{align}
	 \nu_1(p)  &\to -\kappa^2 \simeq  -26.9 , & 
	\tag{\ref{lim-autov-palla-1}} \\
		\label{lim-autof-1} \widetilde\psi_{1,p}^1 &  \to 0   & \mbox{ weakly   in $\mathcal D_{\rad}$ and strongly in $L^{2}_{\loc}(\R^2)$,}
	\intertext{ and, up to an extracted sequence, }
	\label{lim-autof-1-2}
		 \qquad 	\widetilde \psi_{1,p}^2& \to A \,  \eta_1^2 & \mbox{ weakly   in $\mathcal D_{\rad}$ and strongly in $L^{2}_{\loc}(\R^2)$, }
		\end{align}
	for some $A \in \R$,  $A\neq 0$. 	\end{proposition}

\begin{proposition}[Second part of Proposition \ref{prop-3-5}]\label{prop-3-5-2}
	Let $\nu_2(p)$ be the second radial singular negative eigenvalue of \eqref{radial-general-H-no-c} corresponding to $v_p^2$. Then as $p\to\infty$ we have 
	\begin{align}
	 \nu_2(p)&  \to -1 , & 
	\tag{\ref{lim-autov-palla-2}} \\
	\label{lim-autof-2}	
	\widetilde \psi_{2,p}^2& \to 0 	& \mbox{ weakly   in $\mathcal D_{\rad}$ and strongly in $L^{2}_{\loc}(\R^2)$,}
	\intertext{ and, up to an extracted sequence, }
	\label{lim-autof-2-2}	
	\widetilde\psi_{2,p}^1 & \to A \, \eta_1  , & \mbox{ weakly   in $\mathcal D_{\rad}$ and strongly in $L^{2}_{\loc}(\R^2)$,}\end{align}
	for some $A\in \R$,  $A\neq 0$. 
\end{proposition}

The present line of reasoning has many similarities with the one used in \cite{DIPN=2} for the Lane-Emden equation. Indeed Proposition \ref{prop-3-5-1} represents a slight generalization of their arguments, even though the proof that we are going to present  directly uses the singular problems \eqref{radial-general-H-no-c} instead of approximating them with regular Sturm-Liouville problems in collapsing annuli.
Proposition \ref{prop-3-5-2}, instead, is completely new since in the Lane-Emden equation ($\a=0$) the estimate \eqref{stima-autov} is sufficient to see that $\lceil\sqrt{-\nu_2(p)}\rceil = 1$ for any value of $p$ and therefore the contribution of the second eigenvalue to the Morse index is constant.

First we need some estimates that we introduce in a series of lemmas.
As a preliminary we define the function
\begin{equation}\label{def:f}
f_p(r):=p\, r^2|v_p^2(r)|^{p-1}, \quad   0\leq r<1
\end{equation}
and prove some useful properties that descend by the convergence stated in Proposition \ref{prop-3-2} and improve  \cite[Proposition 6.10]{DIPN=2}.

\begin{lemma}\label{lemma-stima-fp} 
	We have \begin{equation}\label{stima-uno}
f_p(r) = p r^2|v_p^2(r)|^{p-1}\leq C \ \text{ for any } r\ge 0 \text{ and } p>1.
\end{equation}
Moreover for any $\rho>0$ there exist $R(\rho)>1$, $K(\rho) >1$ and $p(\rho)>1$ such that for any $R\ge R(\rho)$, $K>K(\rho)$ and $p\geq p(\rho)$
	 \begin{equation}\label{stima-due}
	\max \left\{ f_p(r) \, : \,  r\in [\e_p^1 R,{\e_p^2}/K]\cup [\e_p^2K,1] \right\}\leq 2\rho
	\end{equation}
	where 
	$\e_p^1$ and $\e_p^2$ are as defined in in \eqref{eq:epsilon-2}.
\end{lemma}
\begin{proof}
\eqref{stima-uno} has been obtained in \cite[(2.15)]{DIPN=2}.
	As for \eqref{stima-due},  it can be proved following the line of \cite[Lemma 2.11]{AG-N3}.
Let 
	\[h(s):=W^2(s)s^2=\frac{2(\gamma+2)^2\delta s^{\gamma+2}}{\left(\delta+s^{\gamma+2}\right)^2}\ \text{ and }\  g(s):=W^1(s)s^2=  \frac{64 s^{2}}{\left(8+s^{2}\right)^2}  \]
	where $W^1$ and $W^2$ are as defined in \eqref{limit-W-1-2} and \eqref{limit-W-2-2}.
	For every given $\rho>0$ we choose $K=K(\rho)>1$ such that 
	$h(\frac 1K)<\rho$ and $h(K)<\rho$ and $R=R(\rho)>0$ such that $g(R)<\rho$. This is possible since $h(s)\to 0$ as $s\to 0$,  and $h(s) , g(s)\to 0$ as $s\to\infty$. We let
	\[h_p(s):=f_p(\e_p^2s)=W_{p}^2(s)s^2\ \text{ and }\  g_p(s):=f_p(\e_p^1s)=W_p^1(s)s^2\]
 with $W^i_p$ as in \eqref{rescaled-potential}.
	The convergences in \eqref{limit-W-1-2} and \eqref{limit-W-2-2} imply that 
	$h_p(s)\to h(s)$ uniformly in $[1/K,K]$ and also
	\[f_p\left(\frac {\e_p^2}K\right)= h_p\left(\frac 1K\right)\leq h\left(\frac 1K\right)+\rho<2\rho\]
	\[f_p(\e_p^2K)=h_p(K)\leq h( K)+\rho<2\rho\]
	if $p$ is large enough. Moreover $g_p(s)\to g(s)$ uniformly in $[0,R]$ and also \[f_p(\e_p^1R)=g_p(R)\leq g( R)+\rho<2\rho\]
	if $p$ is large enough.
	\\
			In\cite[Lemma 6.7 and 6.9]{DIPN=2} (see also\cite[Lemma 2.11]{AG-N3})  it is proved that the function $f_p(r)$ has an unique maximum point in each nodal zone of $v_p$, precisely there are $0<c_p<t_{1,p}<d_p<1$ such that $f_p$ is strictly increasing in $(0,c_p)$ and in $(t_{1,p}, d_p)$, while it is strictly decreasing in $(c_p, t_{1,p})$ and in $(d_p, 1)$.\\
Further the convergence of $g_p$ to $g$ in $C^0_{\loc}[0,\infty)$  implies that $c_p\in [0,\e_p^1R]$ if $p$ is large enough, as well as the convergence of 
		$h_p$ to $h$ in $C^0_{\loc}(0,\infty)$ implies that $d_p\in [\frac{\e_p^2}K,\e_p^2K]$. 
		Then the monotonicity properties of $f_p$ yield 
		\begin{align*}
		f_p(r)& <f_p(\e_p^1R)<2\rho \ &  \text{ when } \ r\in [\e_p^1R, t_{1,p}] , \\
		f_p(r) & <f_p \left({\e_p^2}/K\right)<2\rho\ & \text{ when  } \ r\in [t_{1,p},{\e_p^2}/K], \\
		f_p(r)& <f_p(\e_p^2K)<2\rho\ & \text{ when  } \ r\in [\e_p^2K,1]
		\end{align*}
	when $p$ is large enough. This concludes the proof.
\end{proof}

Taking advantage from \eqref{stima-uno} it is not hard to obtain some general estimates, precisely the eigenvalues are bounded and the rescaled eigenfunctions are bounded in $\mathcal D_{\rad}$.

\begin{lemma}\label{lemma-stime-insieme}
	There exist $\bar p >0$ and $C>0$ such that 	for every $p\geq \bar p$ we have 
	\begin{align}\label{eq-stima-autov} 
		-C \le \nu_1(p) <  \nu_2(p)<0
\\ 
\label{eq-stima-norma}	
 \int_0^{\infty} r ((\widetilde\psi_{j,p}^i)')^2\, dr \le C 
\end{align}
for every $i,j=1,2$.
\end{lemma}
\begin{proof}
Using $\psi_{j,p}$ as a test function in \eqref{radial-general-weak-H-no-c} gives
\begin{align}\label{line}
\begin{split}\int_0^1 r\left(\psi_{j,p}'\right)^2& =\int_0^1 r \left( p|v_p^2|^{p-1} +\frac{\nu_j(p)}{r^2} \right)\psi_{j,p}^2 dr \\ &= 
\int_0^1 r^{-1} \left(f_p +\nu_j(p)\right) \psi_{j,p}^2 dr .\end{split}
\end{align}
where $f_p$ is defined in \eqref{def:f}. 
Taking advantage from \eqref{normalization} one can extract $\nu_1(p)$  getting that 
\[
\begin{split}
\nu_1(p) & =\int_0^1 r \left(\psi_{1,p}'\right)^2-r^{-1}
f_p \, \psi_{1,p}^2\, dr\geq-\sup_{r\in (0,1)}f_p(r)
\int_0^1r^{-1}\psi_{1,p}^2 \, dr= -C\end{split}\]
for $p$ large enough, by \eqref{stima-uno}.
Besides, since $\nu_j(p)<0$  for $j=1,2$ by  \eqref{stima-autov}, \eqref{line} also yields that
\[
\int_0^1 r\left(\psi_{j,p}'\right)^2
<  \int_0^1 r^{-1}f_p\psi_{j,p}^2\, dr \le \sup_{r\in(0,1)}f_p(r) \int_0^1 r^{-1}\psi_{j,p}^2\, dr = C .
\]
So also \eqref{eq-stima-norma} is proved, recalling \eqref{stima-uno} and \eqref{normalization}.
\end{proof}	

\begin{lemma}\label{lemma-conv-psi} 
	Let $\widetilde \psi_{j,p}^{i}$ be as defined in \eqref{rescaled-eigenf} for $i,j=1,2$
	and $p_n$ a sequence with $p_n\to \infty$. Then, 
	there exist an extracted sequence (that we still denote by $p_n$), a number $\bar \nu_j\leq 0$
and a function $\widetilde \psi_j^i\in  \mathcal D_{\rad}$ which is a  weak solution to \eqref{eq:finale} with $\beta^i$ substituted by $\bar \nu_j$ such that 
	\[
	\widetilde \psi_{j,p}^{i}\to \widetilde\psi_{j}^{i}  \quad \mbox{weakly in $\mathcal D_{\rad}$ and strongly in $L^{2}_{\loc}(\R^2)$}
	\]
as $p\to\infty$.
\end{lemma}
\begin{proof}
By \eqref{eq-stima-autov} it is clear that there is an extracted sequence $\nu_j(p_n)\to \bar \nu_j\leq 0$. Moreover estimate \eqref{eq-stima-norma} implies that $\widetilde \psi_{j,p}^i$ are uniformly bounded in $\mathcal D_{\rad}$ for $i,j=1,2$.
 Then, up to another extracted subsequence 
 	\[
	\begin{split}
	\widetilde \psi_{j,p_n}^i\to \widetilde\psi_j^i  & \text{ weakly in }  \mathcal D_{\rad} \\
	\widetilde \psi_{j,p_n}^i\to \widetilde\psi_j ^i & \text{ strongly in }  L^{2}(B_R)  \ \forall \  R>0\\
	\widetilde \psi_{j,p_n}^i\to \widetilde\psi_j ^i & \text{ almost everywhere in }\R^2.
	\end{split}\]
In particular $\widetilde\psi_j^i\in {\mathcal D}_{\rad}$ 
	and taking advantage from the fact that the sets $(t_{i-1,p}/\e^i_{p} , t_{i,p}/\e^i_{p})$ invade $(0,\infty)$ by \eqref{limit-zone}, for every $\varphi\in C^{\infty}_0(0,\infty)$ we can choose $n$ so large in such a way that $supp\  \varphi\subset (t_{i-1,p_n}/\e^i_{p_n} , t_{i,p_n}/\e^i_{p_n})$ and $\widetilde \psi_{j,p}^i$ solves
	\[\int_0^\infty r(\widetilde \psi_{j,p}^i)'\varphi' \ dr=\int_0^\infty r W^i_p\widetilde \psi_{j,p}^i\varphi\ dr+\nu_j(p)\int_0^\infty r^{-1}\widetilde \psi_{j,p}^i\varphi\ dr\]
	The weak convergence in $\mathcal D_\rad$ then implies that 
	\[\begin{split}
	& \int_0^\infty r(\widetilde \psi_{j,p}^i)'\varphi' \ dr\to \int_0^\infty r(\widetilde \psi_{j}^i)'\varphi' \ dr\\
	& \int_0^\infty r^{-1}\widetilde \psi_{j,p}^i\varphi\ dr\to \int_0^\infty r^{-1}\widetilde \psi_{j}^i\varphi\ dr
	\end{split}\]
	while the strong convergence in $L^2_{\loc}(B_R)$ and the fact that $W_p^i\to W^i$ in $C^1_{\loc}(0,\infty)$ implies also that 
	\[\int_0^\infty r W^i_p\widetilde \psi_{j,p}^i\varphi\ dr\to \int_0^\infty r W^i\widetilde \psi_{j}^i\varphi\ dr\]
	getting that $\widetilde\psi_j ^{i}$ solves \eqref{eq:finale} in weak sense.
\end{proof}

\begin{remark}\label{remark-dim}
Since the  negative eigenvalues and eigenfunctions of the limit problem \eqref{eq:finale} are known, an immediate consequence of  Lemma \ref{lemma-conv-psi} is that or $\bar \nu_j = -\kappa^2$, $-1$, or $0$ or, else,  $\widetilde\psi^i_j=0$ for $i=1,2$.
Precisely if $\widetilde\psi_j^1\neq0$ then $\bar\nu_j=-1$ or $0$, and similarly if $\widetilde\psi_j^2\neq0$ then $\bar\nu_j=-\kappa^2$ or $0$.
\\
\end{remark}	

	For what concerns the first eigenvalue, the general estimate \eqref{stima-autov} forbids $\bar \nu_1=0$. Next Lemma shows that neither $\bar\nu_1=-1$ is possible
because the limit of first eigenvalue cannot overpass the lowest eigenvalue among the limit problems, which now is $\beta^2_1=-\kappa^2$.

\

\begin{lemma}\label{lemma-stima-primo-autov} 
	We have $\limsup\limits_{p\to \infty} \nu_1(p)\leq  -\kappa^2$.
\end{lemma}
 \begin{proof}
	
It suffices to repeat the proof of Proposition \ref{prop:primo-autovalore-1}  with $\eta^2_1$ instead of $\eta_1$, after choosing $R=R(\e)>\delta^{\frac 1{2k}}$ in such a way $\eta_1^2$ is decreasing in $(R,+\infty)$, it satisfies
$\eta_1^2(r)\leq \eta_1^2(R)<\frac \e4$ for $r>R$ and 
\[\int_0^\infty r^{-1}(\eta_1^2)^2\Phi^2\ dr\geq \int_0^R r^{-1}(\eta_1^2)^2\Phi^2\ dr\geq 1-\e/2\]
since by definition $\displaystyle\int_0^\infty r^{-1}(\eta_1^2)^2 dr=1$.
\end{proof}

 We are now ready to prove Proposition \ref{prop-3-5-1}.
\begin{proof}
[Proof of Proposition \ref{prop-3-5-1}]
Thanks to Lemma \ref{lemma-stima-primo-autov} we know that (up to an extracted sequence) $\nu_1(p)\to \bar\nu_1\le -\kappa^2< -1$.
So the reasoning in Remark \ref{remark-dim} assures that $\widetilde\psi_1^1=0$ and leaves open only two options: or $\widetilde \psi_{1}^{2}=0$,  or, else $\bar \nu_1= -\kappa^2$. In the second case Lemma \ref{lemma-conv-psi} yields that any sequence $p_n\to\infty$ has an extracted subsequence such that  $\widetilde \psi^1_{1,p_{n_k}}\to 0$,  showing \eqref{lim-autov-palla-1} and \eqref{lim-autof-1}. Finally $\widetilde \psi^2_{1,p_{n_k}}\to A \eta^2_1$ for some constant $A\neq 0$, concluding the proof.
\\
Eventually it is left to check that $\widetilde \psi_1 ^{2}\neq 0$.
Let us fix a $\d>0$ such that $\d< \kappa^2/12$ and $R=R(\d)$ and $K=K(\d)$ as in Lemma \ref{lemma-stima-fp}. 
By the definition of $\nu_1(p)$ and by \eqref{normalization} it follows
		\begin{align*}
		-\nu_1(p) & =-\int_0^{1} r \left((\psi_{1,p}')^2-p|v^2_p|^{p-1} (\psi_{1,p})^2\right) dr \le  \int_0^{1} pr |v^2_p|^{p-1} (\psi_{1,p})^2 dr 
		\\
		& = 
		\int_0^{\e_p^1R}pr |v^2_p|^{p-1} (\psi_{1,p})^2  dr +
		\int_{\e_p^1R}^{\frac{\e_p^2}K}pr |v^2_p|^{p-1}  (\psi_{1,p})^2  dr \\
		& +\int_{\frac{\e_p^2}K}^{\e_p^2K}pr |v^2_p|^{p-1}  (\psi_{1,p})^2   dr
		 + \int_{\e_p^2K}^1 pr |v^2_p|^{p-1}  (\psi_{1,p})^2   dr\\
		& =  I_1(p)+I_2(p)+I_3(p)+I_4(p)
		\end{align*}
		Besides, for every $r_0,r_1\in [0,1]$
		\begin{align*}
		\int_{r_0}^{r_1}pr |v^2_p|^{p-1} (\psi_{1,p})^2  dr =\int_{r_0}^{r_1} f_p(r)\frac{(\psi_{1,p})^2}{r} dr
		\le  \max_{r_0<r<r_1}f_p(r)\int _{0}^{1}\frac{(\psi_{1,p})^2}{r}\, dr
		= \max_{r_0<r<r_1}f_p(r),
		\end{align*}
		so the estimate obtained in Lemma \ref{lemma-stima-fp} assures that
		$I_2(p)+I_4(p)<4\delta$ for $p>p(\delta)$.
		\\ 
		For what concerns the first integral, rescaling according to $\ep^1$ gives
		\begin{align*}
		I_1(p)=& \int_0^R rW_p^1(\widetilde\psi^1_{1,p})^2 dr
		\end{align*}
		where $W_p^1\to W^1$ in $C^0_{\loc}[0,+\infty)$ by \eqref{limit-W-1-2} and $\widetilde \psi_{1,p}^1\to 0$ in $L^2_{\loc}(\R^2)$ as noticed before. Then
	there exists $p_2(\d)>0$ such that $I_1(p) < \delta$ if $p>p_2(\delta)$. 
		With respect to third integral, rescaling according to $\ep^2$ gives
		\begin{align*}
		I_3(p)=& \int_{\frac 1K}^K rW_p^2(\widetilde\psi^2_{1,p})^2 dr
		\end{align*}
		where $W_p^2\to W^2$ in $C^0_{\loc}(0,+\infty)$ by \eqref{limit-W-2-2} and $\widetilde \psi_{1,p}^2\to \widetilde \psi_1^2$ in $L^2_{\loc}(\R^2)$ by Lemma \ref{lemma-conv-psi}.
		Then	there exists $p_3(\d)>0$
		such that
		\[I_3(p)\le\int_{\frac 1K}^K rW^2(\widetilde\psi^2_{1})^2 dr+\d \ \text{  for }p>p_3(\d).\]
		Summing up, taking $\bar p=\max\{p(\d), p_2(\d),p_3(\d)\}$ we have
		\[\int_{\frac 1K}^K rW^2(\widetilde\psi^2_{1})^2 dr\geq -\nu_1(p)-6\d \ \text{  for }p>\bar p\]
		and, passing to the $\liminf$  and using Lemma \ref{lemma-stima-primo-autov},
		\[\int_{\frac 1K}^k rW^2(\widetilde\psi^2_{1})^2 dr\geq-\limsup \nu_1(p)-6\d \ge  \kappa^2-6\d> {\kappa^2}/{2} >0  
		\]
		 by the choice of $\delta$. Hence $\widetilde\psi^2_1\neq 0$, concluding the proof. 
\end{proof}

 Eventually we deal with Proposition \ref{prop-3-5-2}.

{\begin{proof}[Proof of Proposition \ref{prop-3-5-2}]
	Recalling \eqref{stima-autov} in order to prove \eqref{lim-autov-palla-2} it suffices to show that for any $\e>0$ there exists $p_\e>1$ such that
	\begin{equation}\label{stima-secondo}
	\nu_2(p)\leq -1+\e\end{equation}	
	 for  $p\geq p_\e$. Let us take a cut-off function $\Phi$ as in \eqref{eq:cut-off}.
	Letting $\e_p^1$ and $\eta_1$ be as defined in \eqref{eq:epsilon-2}  and \eqref{prima-autof-limite-1} respectively, we set
	\begin{equation}\label{psi-epsilon}
	\varphi_p:=\eta_1\left(\frac r{ \e_p^1}\right)\Phi\left(\frac r{ \e_p^1}\right) +a_p \psi_{1,p} \ \ \text{ as }r\in[0,1]
	\end{equation}
	where $R=R(\e)>0$ is by now fixed and satisfies \eqref{R-e-1},  \eqref{R-e-2}, while 
  $a_p\in \R$ is such that $\varphi_p\underline{\perp} \psi_{1,p}$  according to \eqref{orto}, namely
	\begin{align}\nonumber
	a_p:& =-\frac{ \int_0^1 r^{-1} \eta_1\left(\frac r{ \e_p^1}\right)\Phi\left(\frac r{ \e_p^1}\right)  \psi_{1,p}(r)\, dr }{\int_0^1 r^{-1}(\psi_{1,p})^2\, dr   } \\
 \label{a-p}	& \underset{\eqref{normalization}}{=}-\int_0^1 r^{-1} \eta_1\left(\frac r{ \e_p^1}\right)\Phi\left(\frac r{ \e_p^1}\right)   \psi_{1,p}(r)\, dr .
	\end{align}
	Notice that since $\e_p^1\to 0$ we may assume w.l.g. that $p$ is so large that $\frac 1{ \e_p^1}>2R$, so that $\varphi_p\in \mathcal H_{0,\rad}$.\\	
	We insert the test function $\varphi_p$ into the variational characterization \eqref{nu-var-i} of $\nu_2(p)$ and get
	\begin{equation}\label{secondo-autov}
	\nu_2(p)\leq \frac{\int_0^1r \left((\varphi_p')^2- p|v_p^2|^{p-1}\varphi_p ^2\right)\, dr} {\int_0^1r^{-1}\varphi_p^2\, dr},
	\end{equation}
	then we  compute all the terms.
		
\

	\noindent First we claim that $a_p\to 0 $ as $p\to \infty$.
	Indeed we can write
\begin{align*}
	&\int_0^1 r^{-1}\eta_1\left(\frac r{ \e_p^1}\right)\Phi\left(\frac r{ \e_p^1}\right)    \psi_{1,p}(r)\, dr=
	\int_0^{t_{1,p}} r^{-1}\eta_1\left(\frac r{ \e_p^1}\right)\Phi\left(\frac r{ \e_p^1}\right)    \psi_{1,p}(r)\, dr\\
	&
	+\int_{t_{1,p}}^1 r^{-1}\eta_1\left(\frac r{ \e_p^1}\right)\Phi\left(\frac r{ \e_p^1}\right)    \psi_{1,p}(r)\, dr
	\end{align*}	
	Rescaling with respect to  $\e_p^1$ we have that
	\begin{align*}
	\int_0^{t_{1,p}} r^{-1}\eta_1\left(\frac r{ \e_p^1}\right)\Phi\left(\frac r{ \e_p^1}\right)    \psi_{1,p}(r)\, dr
	=\int_0^{\frac {t_{1,p}}{\e_p^1}} s^{-1} \eta_1( s)\Phi(s)    \widetilde\psi_{1,p}^{1}(s) \, ds  \\
	\intertext{ and since $\frac {t_{1,p}}{\e_p^1}\to \infty$ by \eqref{limit-zone}, recalling that the support of $\Phi$ is compact we get}
	=\int_0^{ \infty} s^{-1} \eta_1( s)\Phi(s)    \widetilde\psi_{1,p}^{1}(s)\, ds
	\to 0
	\end{align*}
	as $p\to \infty$, because  $\widetilde\psi_{1,p}^{1}\to 0$ weakly in 
	$\mathcal L_{\rad}(\R^2)$ by Proposition \ref{prop-3-5-1}.
Further the same property \eqref{limit-zone} implies that
	\[\int_{t_{1,p}}^1 r^{-1}\eta_1\left(\frac r{ \e_p^1}\right)\Phi\left(\frac r{ \e_p^1}\right)    \psi_{1,p}(r)\, dr=0\]
for  $p$ so large that  $\frac {t_{1,p}}{\e_p^1}>2R$.
Next

\begin{align}\nonumber
\int_0^1r |\varphi_p'|^2 dr  & =  \int_0^1r \left(\left(\eta_1\left(\frac{r}{\e_p^1}\right)\Phi\left(\frac r{ \e_p^1} \right)\right)'\right)^2 dr \\ \nonumber
& + a_p^2\int_0^1r(\psi_{1,p}')^2\ dr+2a_p\int_0^1r\psi_{1,p}'\left(\eta_1^1\left(\frac{r}{\e_p^1} \right)\Phi^2\left(\frac r{ \e_p^1} \right)\right)' dr
\intertext{ and since $\e_p^1\to 0$, the same  computations made to obtain \eqref{f4} in the proof of Proposition \ref{prop:primo-autovalore-1} give}
\nonumber 
& = -\int_0^{\infty} s^{-1}(\eta_1)^2\Phi^2 ds  +\int_0^{\infty}s\, W^1 (\eta_1)^2 \Phi ^2 ds
+ \int_0^{\infty} s (\eta_1)^2 (\Phi' )^2 ds \\
\label{f4bis}
& + a_p^2\int_0^1r(\psi_{1,p}')^2\ dr+2a_p\int_0^1r\,\psi_{1,p}'\left(\eta_1\left(\frac{r}{\e_p^1}
\right)\Phi^2\left(\frac r{ \e_p^1} \right)\right)' 
\ dr
\end{align}

Moreover
	\begin{align} \nonumber
	\int_0^1r p|v_p^2|^{p-1}\varphi_p ^2\, dr& =\int_0^1 r p|v_p^2|^{p-1}\left(\eta_1\left(\frac r{ \e_p^1}\right)\Phi\left(\frac r{ \e_p^1}\right)\right)^2dr
	 + a_p^2\int_0^1r p|v_p^2|^{p-1}\psi_{1,p}^2 dr\\ \nonumber
	 & + 2 a_p\int_0^1 r p|v_p^2|^{p-1}\psi_{1,p} \eta_1\left(\frac r{ \e_p^1}\right)\Phi\left(\frac r{ \e_p^1}\right)dr
	\intertext{and rescaling with respect to $\e_p^1$ in the first integral, since $\frac 1{\e_p^1}>2R$ we get} \nonumber 
	&=\int_0^{\infty}sW^1_p(\eta_1)^2\Phi^2ds +a_p^2\int_0^1r p|v_p^2|^{p-1}\psi_{1,p}^2 dr \\
	\label{step3bis} & 	+ 2 a_p\int_0^1 r p|v_p^2|^{p-1}\psi_{1,p} \eta_1\left(\frac r{ \e_p^1}\right)\Phi\left(\frac r{ \e_p^1}\right)dr.
	\end{align}

\

Putting together \eqref{f4bis} and \eqref{step3bis}  we obtain
\begin{align*}
&\int_0^1r |\varphi_p'|^2 dr
-\int_0^1r p|v_p^2|^{p-1}\varphi_p ^2\, dr = -\int_0^{\infty} s^{-1}(\eta_1)^2\Phi^2 \ ds  \\
&+\int_0^{\infty}s\left[W^1-W^1_p\right] (\eta_1)^2 \Phi^2  ds
+  \int_0^{\infty} s (\eta_1)^2 (\Phi')^2 ds\\
&+a_p^2\int_0^1 r\left((\psi_{1,p}')^2     -p|v_p^2|^{p-1}\psi_{1,p}^2    \right)   dr\\
&+2a_p \int_0^1 r\left(\psi_{1,p}'\left(\eta_1\left(\frac{r}{\e_p^1}
\right)\Phi\left(\frac r{ \e_p^1} \right)\right)' 
-p|v_p^2|^{p-1}\psi_{1,p} \eta_1\left(\frac r{ \e_p^1}\right)\Phi\left(\frac r{ \e_p^1}\right)
\right)    dr.
\end{align*}
Since $\psi_{1,p}$ solves \eqref{radial-general-H-no-c} we get
\begin{align*}
\int_0^1 r\left((\psi_{1,p}')^2     -p|v_p^2|^{p-1}\psi_{1,p}^2    \right)   dr = \nu_1(p) \int_0^1 r^{-1} \psi_{1,p}^2  dr \underset{\eqref{normalization}}{=} \nu_1(p),
\end{align*}	
	and similarly
	\begin{align*}
\int_0^1 r\left(\psi_{1,p}'\left(\eta_1\left(\frac{r}{\e_p^1}
\right)\Phi\left(\frac r{ \e_p^1} \right)\right)' 
-p|v_p^2|^{p-1}\psi_{1,p} \eta_1\left(\frac r{ \e_p^1}\right)\Phi\left(\frac r{ \e_p^1}\right)
\right)    dr \\
=  \nu_1(p)\int_0^1 r^{-1}  \eta_1\left(\frac{r}{\e_p^1}
\right)\Phi\left(\frac r{ \e_p^1} \right) \psi_{1,p} dr \underset{\eqref{a-p}}{=} -\nu_1(p) a_p .
	\end{align*}	
Eventually
\begin{align} \nonumber 
\int_0^1r |\varphi_p'|^2 dr
-\int_0^1r p|v_p^2|^{p-1}\varphi_p ^2\, dr= -\int_0^{\infty} s^{-1}(\eta_1)^2\Phi^2 ds  
\\ \label{num}
+\int_0^{\infty}s\left[W^1-W^1_p\right] (\eta_1)^2 \Phi^2  ds+  \int_0^{\infty} s (\eta_1)^2 (\Phi')^2 ds 
 - \nu_1(p) a_p^2
\end{align}

\

On the other hand  by the definition of $\varphi_p$ it follows 
	\begin{align} \nonumber
	\int_0^1r^{-1}\varphi_p^2, dr &=\int_0^1 r^{-1}\left(\eta_1\left(\frac r{ \e_p^1} \right)\Phi\left(\frac r{ \e_p^1} \right)\right)^2 + \\ \nonumber
	 & + a_p^2\int_0^1 r^{-1}(\psi_{1,p})^2+2 a_p\int_0^1 r^{-1} \eta_1\left(\frac r{ \e_p^1}\right)\Phi\left(\frac r{\e_p^1}\right)\psi_{1,p}(r) \\ \nonumber 
	&  \underset{\eqref{normalization}, \eqref{a-p}}{=} \int_0^1r^{-1}\left(\eta_1\left(\frac r{ \e_p^1} \right)\Phi\left(\frac r{ \e_p^1} \right)\right)^2 - a_p^2 
	\intertext{and rescaling with respect to $\e_p^1$ and using the properties of $\Phi$}
	\label{den}
	& = \int_0^{\infty} s^{-1}(\eta_1)^2\Phi^2  ds  - a_p^2.
	\end{align}

Inserting \eqref{num} and \eqref{den} into \eqref{secondo-autov}  we obtain
\begin{align*}
	\nu_2(p)\leq &  \dfrac{-\int_0^{\infty} s^{-1}(\eta_1)^2\Phi^2 ds  +\int_0^{\infty}s\left[W^1-W^1_p\right] (\eta_1)^2 \Phi^2  ds
	+  \int_0^{\infty} s (\eta_1)^2 (\Phi')^2 ds - \nu_1(p) a_p^2}{\int_0^{\infty} s^{-1}(\eta_1)^2\Phi^2  ds- a_p^2} \\
 & = -1 + \dfrac{\int_0^{\infty}s\left[W^1-W^1_p\right] (\eta_1)^2 \Phi^2  ds	+  \int_0^{\infty} s (\eta_1)^2 (\Phi')^2 ds- \left(\nu_1(p)+1\right) a_p^2}{ \int_0^{\infty} s^{-1}(\eta_1)^2\Phi^2  ds- a_p^2}.
 \end{align*}
By the choice of $R$ we have
\begin{align*} 
&	\int_0^{\infty} s (\eta_1)^2(\Phi')^2 ds=\int_R^{2R} s (\eta_1)^2(\Phi'))^2 ds\\
&	 \le \frac{4}{R^2}\int_R^{2R}  s \eta_1^2(s)ds \underset{{ \eqref{R-e-1}}}{<} \frac{\e^2 }{4 R^2}  \int_R^{2R} s\  ds = \frac{3\e^2}{8}  < \frac{3\e}{8} .
\end{align*} 

Therefore using also \eqref{R-e-2} we get
\begin{align*} 
 \nu_2(p)<  & -1 + \dfrac{\int_0^{\infty}s\left|W^1-W^1_p\right| (\eta_1)^2 \Phi^2  ds	+ \frac{3\e}{ 8}  - \left(\nu_1(p)+1\right) a_p^2}{ 1-\frac{\e}{2} - a_p^2}.
\end{align*}
Since $a_p\to 0$ and $\nu_1(p)$ is bounded, we can choose $p=p_{\e}$  so large that 
\[ a^2_p<\frac \e2 \ \  \text{ and } \ \ \left(\nu_1(p)+1\right) a_p^2 > -\frac{\e}{16} .\]
	On the other hand by the properties of $\Phi$  we have
\begin{align*}
\int_0^{+\infty} s  \left|W^1 -W^1_p\right| (\eta_1)^2\Phi^2 ds \le \sup_{(0,2R)}|W^1-W^1_p|\int_0^{2R}s \eta_1^2 ds 
\end{align*}
and since $W^1_p\to W^1$ uniformly on $[0,2R]$ we can possibly enlarge  $p_{\e}$ in such a way that
\[ \sup_{(0,2R)}|W^1-W^1_p| \le \frac{\e }{16\int_0^{2R}s \eta_1^2 ds }  \quad \text{ for } p> p_{\e}.\]
Eventually we end up with  
\[\nu_2(p) < - 1 + \frac{\e}{ 2(1-\e)} < -1+\e,\] which concludes the proof of \eqref{lim-autov-palla-2}.
Next Lemma \ref{lemma-conv-psi} and Remark \ref{remark-dim} yield also \eqref{lim-autof-2} and \eqref{lim-autof-2-2}. \\
In particular   $\widetilde{\psi}_{2,p}^1\to A\eta_1$ weakly in $\mathcal D_\rad$ and strongly in $L^2_{\loc}(\R^2)$ for some $A\in \R$. It remains to show that $ A\neq 0$, which can be seen reasoning as in the proof of Proposition \ref{prop-3-5-1}.
Recalling the definition of $\nu_2(p)$ and the normalization in \eqref{normalization} we have 
	\begin{align*} -\nu_2(p) & = -\int_0^1 r\left(\left(\psi_{2,p}'\right)^2-p|v_p^2(r)|^{p-1}\psi_{2,p}^2\right) dr \le \int_0^1 r p|v_p^2(r)|^{p-1}\psi_{2,p}^2dr 
	\intertext{For any $\e>0$ we choose $R=R(\e)$ and $K=K(\e)$ as in Lemma \ref{lemma-stima-fp} and we divide the interval $(0,1)$ in the following way  }
& = \int_0^{\e_p^1R}  r p|v_p^2(r)|^{p-1}\psi_{2,p}^2 dr+ \int_{\e_p^1R}^{\frac{\e_p^2}K}  r p|v_p^2(r)|^{p-1}\psi_{2,p}^2 dr  \\
&  +\int_{\frac{\e_p^2}K}^{\e_p^2K}  r p|v_p^2(r)|^{p-1}\psi_{2,p}^2 dr+ 	\int_{\e_p^2K}^1 r p|v_p^2(r)|^{p-1}\psi_{2,p}^2 dr \\
	&=I_1(p)+I_2(p)+I_3(p)+I_4(p)
	\end{align*}
	{By the same computations made in the proof of Proposition \ref{prop-3-5-1},
	Lemma \ref{lemma-stima-fp} and the normalization in \eqref{normalization} imply that there exists $p_\e$ such that
		\begin{equation}\label{f5}
	\begin{split}
	I_4(p)=&\int_{\e_p^2K}^1rp|v_p^2(r)|^{p-1}\psi_{2,p}^2(r)\, dr  \le  \max_{\e_p^2K<r<1}f_p(r)\int _{\e_p^2K}^1 r^{-1}\psi_{2,p}^2(r)\, dr \\
	&\le  \max_{\e_p^2K<r<1}f_p(r) <\e
	\end{split}\end{equation}
	and in the same manner 
	\begin{equation}\label{f5-bis}
	I_2(p)\le \max_{     \e_p^1R <r< \frac{e_p^2} K  }f_p(r) <\e
	\end{equation}
	}
\remove{	Rescaling the integral $I_3(p)$ we have instead,  denoting by $\widetilde\psi_{2}^2$ the weak limit in $\mathcal D_{\rad}$ of $\widetilde\psi_{2,p}^2$  as $p\to \infty$
	\[\begin{split}
	I_3(p)= &  \int_{\frac 1K}^{K} sW_p^2(s)\left(\widetilde\psi_{2,p}^2\right)^2\, ds\\
	&=      \int_{\frac 1K}^{K} sW^2\left(\widetilde\psi_{2}^2\right)^2\, ds+   \int_{\frac 1K}^{K}s\left(W^2_p\left(\widetilde\psi_{2,p}^2\right)^2-W^2\left(\widetilde\psi_{2}^2\right)^2\right) ds\\
	&= \int_{\frac 1K}^{K}s\left(W^2_ p\left(\widetilde\psi_{2,p}^2\right)^2-W^2\left(\widetilde\psi_{2}^2\right)^2\right) ds
	\end{split}\]
	since $\widetilde\psi_{2}^2\equiv 0$  by \eqref{lim-autof-2}. Moreover by \eqref{limit-W-2-2} and Lemma \ref{lemma-conv-psi} we know that $W_p^2(s)$ converges uniformly in $[\frac 1K,K]$ to $W^2(s)$, while $\widetilde\psi_{2,p}^2\to \widetilde\psi_{2}^2$ in $L^2(\frac 1K,K)$ as $p\to \infty$. This means that there exists an exponent $\tilde p_\e$ such that for any $p\geq \tilde p_\e$ it holds
	\begin{equation}\label{stima-I-3}
	I_3{\taglia \geq -\ } {\AL \le } 2\e.
	\end{equation}
	\edz{\F La frase in rosso sostituisce quello prima a partire da Rescaling }}
{Rescaling the integral $I_3(p)$ we have instead
	\[\begin{split}
	I_3(p)= &\int_{\frac{\e_p^2}K}^{\e_p^2K}  r p|v_p^2(r)|^{p-1}\psi_{2,p}^2 dr=\int_{\frac 1K}^{K} sW_p^2(s)\left(\widetilde\psi_{2,p}^2\right)^2\, ds\\
	=& \int_{\frac 1K}^{K}s\left[W_p^2(s)-W^2(s)\right]\left(\widetilde\psi_{2,p}^2\right)^2\, ds+\int_{\frac 1K}^{K} sW^2(s)\left(\widetilde\psi_{2,p}^2\right)^2\, ds\\
&\le \sup_{\frac 1K<|x|<K}	\left|W_p^2(s)-W^2(s)\right|+C\int_{\frac 1K}^{K} s\left(\widetilde\psi_{2,p}^2\right)^2\, ds
	\end{split}\]
	by \eqref{normalization} and the boundedness of $W^2$. Next \eqref{limit-W-2-2} implies that $W_p^2$ converges uniformly in $[\frac 1K,K]$ to $W^2$, while $\widetilde \psi_{2,p}^2\to 0$ in $L^2(B_K)$, for every $K$, as $p\to \infty$ by \eqref{lim-autof-2} and Lemma \ref{lemma-conv-psi}, showing that 
	$I_3(p)\to 0 $  as $p\to \infty$ and there exists an exponent $\tilde p_\e$ such that for any $p\geq \tilde p_\e$ it holds
	\begin{equation}\label{stima-I-3}
	I_3(p) \le  2\e.
	\end{equation}
	}
	Finally, rescaling the integral $I_1(p)$ we have
	{
	\[\begin{split}
	I_1(p)=& \int_0^{\e_p^1R}rp|v_p^2|^{p-1}\psi_{2,p}^2\, dr=\int_0^R sW_p^1\left(\widetilde \psi_{2,p}^1\right)^2\, ds\\
	&=\int_0^R s\left( W^1_p- W^1\right)\left(\widetilde \psi_{2,p}^1\right)^2 ds+
	\int_0^R s W^1\left( \left(\widetilde \psi_{2,p}^1\right)^2-\left(A\eta_1\right)^2\right) ds+ A^2\int_0^Rs W^1(\eta_1)^2 \ ds\\
	&\leq \sup_{|x|<R}\left|W^1_p(s)- W^1(s)\right|+C\int_0^R s\left|\left(\widetilde \psi_{2,p}^1\right)^2-\left(A\eta_1\right)^2\right| ds + A^2\int_0^Rs W^1(\eta_1)^2 \ ds
	\end{split}\]
	by \eqref{normalization} and the boundedness of $W^1$.
	Next	 \eqref{limit-W-1-2} implies that $W_p^1$ converges uniformly in $[0,R]$ to $W^1$ while $\widetilde\psi_{2,p}^1\to  A \eta_1$ in $L^2(B_R)$ as $p\to \infty$ by \eqref{lim-autof-2-2},
		   showing that there exists an exponent $\hat p_\e$ such that for any $p\geq \hat p_\e$ it holds
	\begin{equation}\label{stima-I-1}
	 I_1(p) \le A^2  \int_0^R sW^1 (\eta_1)^2 ds + 2\e.
	\end{equation}
	}
	Finally from \eqref{lim-autov-palla-1} we have that for any $\e>0$ there exists $p^*_\e>1$ such that  for any $p\geq  p^*_\e$ it holds
	\begin{equation}\label{stima-nu-2}
	\nu_2(p)\leq -1+\e.
	\end{equation}
	Choosing  $p\geq \max\{ p_\e,\tilde p_\e, \hat p_\e, p^*_\e\}$ then \eqref{f5}, \eqref{f5-bis}, \eqref{stima-I-3} and \eqref{stima-I-1} imply that 
	\begin{equation}
	 1-\e< - \nu_2(p) \le  A^2 \int_0^R sW^1 (\eta_1)^2 ds +  6\e
	\end{equation}
	giving
	\[  A^2 \int_0^R sW^1 (\eta_1)^2 ds  \geq 1-7\e>0\]
	for  $\e<\frac 17$. This implies that  $A\neq 0$ and concludes the proof.
\end{proof}

\

\section{Least energy solutions in symmetric spaces}\label{sec:4}

	In this section we want to find new solutions to \eqref{H} which  admit some rotational symmetry.
 To this end, for any angle $\psi$, we denote by $R_\psi$ the rotation of angle $\psi$ in counterclockwise direction centered at the origin and by $\mathcal G_\psi$ the subgroup of $SO(2)$ generated by $R_\psi$. In particular we consider angles $\psi=\frac {2\pi}n$ with $n\in \N$, $n\geq 1$, so that $\mathcal G_{\frac {2\pi}n}$ is a proper subgroup of $SO(2)$. \\
We say that 
a function $u$ defined in $B$ is $n$-invariant if it satisfies
\begin{equation}\label{k-inv}
v(x)=v\left(g(x)\right) \ \ \text{ for every }x\in B, \ \ \text{ for every } g\in \mathcal G_{\frac {2\pi}n}.
\end{equation}
Next we denote by $H^1_{0,n}$ the subspace of $H^1_0(B)$ given by functions which are $n$-invariant, namely 
\begin{equation}\label{eq:def-H-k}
H^1_{0,n}:=\{v\in H^1_0(B) :  \ v(x)=v(g(x)) \ \ \text{ for any }x\in B,  \text{ for every } g\in \mathcal G_{\frac {2\pi}n}\}.
\end{equation}

For $n=1$  $\mathcal G_{2\pi}=\{I\}$ is the trivial subgroup of $SO(2)$, so  the space $H^1_{0,1}$ coincides with $H^1_0(B)$, while all the other spaces $H^1_{0,n}$ are strictly contained in $H^1_0(B)$. 
Observe also that $\mathcal G_{\frac {2\pi}n}$ is a subgroup of $\mathcal G_{\frac {2\pi}m}$ if $m$ is a multiple of $n$ showing that
$H^1_{0,m}\subseteq H^1_{0,n}$ in this case.   Lastly $H^1_{0,\rad}\subset H^1_{0,n}$ for every $n$.

 In order to obtain new $n$-invariant solutions let us recall for a while how positive and sign changing solutions to \eqref{H} can be produced when the problem has a variational structure, as in our case, namely 
when solutions are critical points for the energy functional
 $\mathcal E(u)$ as defined in  \eqref{eq:energy}. It is standard, in this situation, to find solutions looking at the minima of $\mathcal E(u)$ constraint on the manifold 
 \[ \mathcal N:=\{v\in H^1_0(B): \text{ s.t. }v\neq 0, \ \mathcal E'(v)v=0\}\]
 where $\mathcal E'(u)$ denotes the Fr\'echet derivative of $\mathcal E$ in $u$.  In order to find sign changing solutions, instead, the nodal Nehari manifold has been introduced, see \cite{CCN} and \cite{BWW} and nodal solutions can be found looking at the minima of $\mathcal E(u)$ on the manifold 
 \[\mathcal N_{\nod}:=\{v\in H^1_0(B): \text{ s.t. }v^+, \, v^- \neq 0, \ \mathcal E'(v)v^+=0, \ \mathcal E'(v)v^+=0\}\]
 where $s^+$ and $s^-$ denote the positive and the negative part of $s$ respectively.  
 As an example of how this procedure can be performed to obtain solutions to the H\'enon problem we quote the paper \cite{SerraTilli} where the authors proved the existence of a positive and radially increasing solution to the H\'enon problem with Neumann boundary conditions in dimension $N\geq 2$ and \cite{BWW} and \cite{BW} that deal with nodal solutions.\\
 The same construction can be repeated in the symmetric spaces $H^1_{0,n}$ after
  introducing,  for every $n\geq 1$,  the $n$-invariant Nehari manifold 
 \[\mathcal N_n:=\{v\in H^1_{0,n}\, : \, v\neq 0,  \   \ \mathcal E'(v)v=0 \}\]
and the nodal  $n$-invariant  Nehari manifold
 \[ \mathcal N_{n, \nod}=\Big\{v\in H^1_{0,n} \, : \,  v^+, \, v^- \neq 0 ,   \ \mathcal E'(v)v^+=0, \ \mathcal E'(v)v^+=0  \Big\}\]
Since, for every $p>1$,  $H^1_{0,n}$, is compactly embedded in  $L^p(B)$,
it is quite standard to see that $\min_{u\in \mathcal N_n}\mathcal E(u)$ is attained at a nontrivial function, that we denote by $u_{p,n}^1$   and call {\it least energy $n$-nvariant} solution. By the principle of symmetric criticality in \cite{Palais}
these
functions $ u_{p,n}^1\in H^1_{0,n}$ are symmetric critical points for $\mathcal E(u)$
and hence, are weak solutions to \eqref{H}  that are positive in $B$ by construction. 
\\
 In a similar way also $\min_{u\in \mathcal N_{n, \nod}}\mathcal E(u)$ is attained at a nontrivial function
that we denote by $u^2_{p,n}$ and call {\it least energy nodal $n$-invariant solution}.
 Again the principle of symmetric criticality shows that $u^2_{p,n}$ are weak solutions to \eqref{H}  that change sign in $B$ by construction.
 \\
When $n=1$ $u^1_{p,1}$ and $u^2_{p,1}$ coincide with the least energy and the nodal least energy solutions to \eqref{H}  that have been studied in \cite{AdG} and \cite{GGP13}.

In the remaining of this section we shall prove that, for suitable values of the integer $n$, such least energy energy  $n$-invariant solutions are nonradial and distinct one from another, thus obtaining the multiplicity results stated in the Introduction as Theorems \ref{teo:existence-1} and \ref{teo:existence-2}. Non-radiality will be proved  by considerations based on the Morse index  in the spaces $H^1_{0,n}$, while the fact that such solutions do not coincide follows by a strict monotonicity result in \cite{Gladiali-19}. 
The present multiplicity result is inspired by an analogous one in  \cite{GI}, dealing with { nodal solutions to }the Lane-Emden problem. 
In this last paper the spaces $H^1_{0,n}$ are slightly different since the functions in \cite{GI} have to be symmetric with respect to one variable. Basically in \cite{Gladiali-19} it is shown that solutions in $H^1_{0,n}$, under some additional assumption which is satisfied in the present situation, 
are symmetric with respect to a direction in a sector of amplitude $\frac{2\pi}n$, so that the solutions in \cite{GI} are, up to a rotation, the same we found here working in $H^1_{0,n}$ without imposing an extra symmetry. 
 Note that while in \cite{GI} this procedure produces results only in the case of nodal solutions, in the framework of the H\'enon problems it finds a wider range of applications.

In the following subsection we define the notion of Morse index in the symmetric spaces $H^1_{0,n}$ and we compute it for the least energy solutions $u^1_{p,n}$ and $u^2_{p,n}$  by taking advantage from their minimality. Next, using the asymptotic results obtained in Section \ref{sec:3}, we compute it also for the radial solutions $u^1_p$ and $u^2_p$, when the parameter $p$ is large. Eventually in the last subsection we prove the multiplicity results.

\subsection{The $n-$symmetric Morse index}\label{se:3.3}

Working in the symmetric spaces $H^1_{0,n}$, $n\geq 1$ we need to adapt the notion of Morse index to these spaces. To this end, if $u_p$ is a solution to \eqref{H} that belongs to $H^1_{0,n}$ we denote by $m_n(u_p)$ the Morse index of $u_p$ in the space $H^1_{0,n}$, namely the
maximal dimension of a subspace $X$ of $H^1_{0,n}$ in which the quadratic form $Q_u$ is negative defined, or equivalently, the
 number of negative eigenvalues of the linearized operator $L_{u_p}$ which have corresponding eigenfunction in $H^1_{0,n}$. We refer hereafter to $m_n(u_p)$ as the $n$-invariant Morse index of $u_p$.

 Following  \cite{BW} it is not hard to see that when $u_p=u^1_{p,n}$ or $u_p=u^2_{p,n}$ are the least energy $n$-invariant (or least energy nodal $n$-invariant) solutions to \eqref{H}, then 
\begin{equation}\label{morse-index-n-symmetric}
m_n(u^1_{p,n})=1 \ \ \text{ and } \ m_n(u^2_{p,n})=2 .
\end{equation}
Indeed 
	a minimum $u^1_n$ of $\mathcal E(u)$ on $\mathcal N_n$ satisfies $\langle\mathcal E''(u^1_n)\psi,\psi\rangle\geq 0$ for every $\psi$ on the tangent space to $\mathcal N_n$, where $\mathcal E''(u)$ is the second Fr\'echet derivative of $\mathcal E$ at $u$  and $\langle,\rangle$ is the pairing.
		Since $\langle\mathcal E''(u)\psi,\psi\rangle=Q_u(\psi)$, where $Q_u$ is the quadratic form as in \eqref{eq:Q-u}, and $\mathcal N_n$ has codimension 1, this shows that $m_n(u^1_n)\leq 1$.
		The fact that the $n$-Morse index of $u^1_n$ is exactly one then follows observing that, since $u^1_n\in \mathcal N_n$, we have
		\[Q_{u^1_n}(u^1_n)=\int _B |\nabla u^1_n|^2-p\int_B|x|^\a|u^1_n|^{p+1}=(1-p)\int _B |\nabla u^1_n|^2<0 .\]
		In the same way a minimum $u^2_n$ of $\mathcal E(u)$ on $\mathcal N_{n,\nod}$  satisfies $\langle\mathcal E''(u^2_n)\psi,\psi\rangle\geq 0$ for every $\psi$ on the tangent space to $\mathcal N_{n,\nod}$. Since $\mathcal N_{n,\nod}$ has codimension 2 in $H^1_{0,n}$, it follows that $m_n(u^2_n)\leq 2$.
		Besides both the positive and negative part of $u^2_n$ belong to $\mathcal N_n$, so that  
		\[Q_{u^2_n}((u^2_n)^\pm)=\int _B |\nabla (u^2_n)^\pm|^2-p\int_B|x|^\a|(u^2_n)^\pm|^{p+1}=(1-p)\int _B |\nabla (u^2_n)^\pm|^2<0 ,\]
		which proves that $m(u^2_n)=2$ because $(u^2_n)^+$ and $(u^2_n)^-$ are linearly independent
		 and concludes the proof of \eqref{morse-index-n-symmetric}.\\

Generally speaking, for any solution $u_p$ to \eqref{H}  that belongs to $H^1_{0,n}$ \cite[Proposition 3.7]{AG-sez2} states that $m_n(u_p)$ coincides with the number of negative singular eigenvalues $\widehat \L_h(p)$ of \eqref{singular-eigenvalue-problem} which have corresponding eigenfunction in $H^1_{0,n}$. 
Coming to radial solutions, as in the case of the functions $u_p^1$ and $u_p^2$ studied before, the $n$-invariant Morse index can be computed starting from the decomposition 
recalled in Section \ref{sec:3}:
 \begin{equation*} \tag{\ref{decomposition}}
\widehat \L _h(p)=\widehat \L_j^{\rad}(p)+k^2,
\end{equation*}
and the shape of the corresponding eigenfunctions 
 \begin{equation*}\tag{\ref{eq:autofunzioni}}
\widehat \phi_h(r,\theta)=\widehat \phi_j^\rad (r)(A \cos (k \theta)+B\sin (k \theta)).
\end{equation*}
Indeed $ \widehat \L_h(p)$ (when it is negative) has a corresponding eigenfunction in $H^1_{0,n}$ if and only if either $k=0$ in \eqref{decomposition}, since the corresponding eigenfunction is radial so that belongs to $H^1_{0,n}$ for every $n$, or, else if $A\cos (k\theta)+B\sin(k\theta)$ belongs to $H^1_{0,n}$, namely when $k$ is a multiple of $n$ (i.e. $k/n$ is an integer). In particular the multiplicity of $ \widehat \L_h(p)$,  when it is not zero,  is either $1$ corresponding to $k=0$ in \eqref{decomposition}, or $2$ corresponding to $k= l n$ for some positive integer $l$.\\
Thanks to this in the present setting the general formula for the symmetric Morse index \cite[Corollary 5.10]{AG-sez2} becomes
\begin{align}\label{eq:Morse-simmetrico}
		m_n(u^i_p) & = i+ 2 \sum\limits_{j=1}^{i} \left[ \frac 1n \left\lceil \frac{2+\a}{2} \sqrt{- {\nu}_j(p) }-1\right\rceil \right]
		\end{align}
		where $\nu_j(p)$ are the negative singular eigenvalues associated with $v_p^1$ and $v_p^2$ that have been studied in Section \ref{sec:3}, 
		$\lceil \, \rceil$ and $[\, ]$ stand respectively for the ceiling function and the integer part.
	\\
We can then deduce from the asymptotic behaviour of $\nu_j(p)$ obtained in Propositions \ref{prop:primo-autovalore-1} and  \ref{prop-3-5} the value of the $n$-Morse index of 	$u_p^1$ and $u_p^2$  for large values of $p$.

	\begin{corollary}\label{cor:morse-sim}
		Let $\a\geq 0$ be fixed and let $u^1_p$  and $u_p^2$ be a least energy radial and a least energy nodal radial solution to \eqref{H} corresponding to $\a$ respectively. Then there exists $p^\star=p^\star(\a)>1$ and $p_2^\star =p_2^\star(\a)>1$ such that for any $p>p^\star$ we have
\[
	m_n(u_p^1)  = 1+2\left[ \frac 1n \left\lceil\frac{\alpha}{2}\right\rceil\right]
\]
	and for any $p>p^\star_2$ we have
\[
	m_n(u_p^2)  =2+ 2\left[ \frac 1n  \left\lceil\frac{2+\alpha}{2}\kappa-1\right\rceil  \right]+2\left[ \frac 1n \left\lceil\frac{\alpha}{2}\right\rceil\right]
	\]
if $\a\neq \a_i= 2(\frac{i}{\kappa}-1)$, or
	\[
	{2+ 2\left[ \frac 1n  \left\lceil\frac{2+\alpha}{2}\kappa\right\rceil  \right]+2\left[ \frac 1n \left\lceil\frac{\alpha}{2}\right\rceil\right] 
	\ge m_n(u_p^2) \ge 2+ 2\left[ \frac 1n  \left\lceil\frac{2+\alpha}{2}\kappa-1\right\rceil  \right]+2\left[ \frac 1n \left\lceil\frac{\alpha}{2}\right\rceil\right]}
\]
	if $\a=\a_i$ for some integer $i$.
\remove{	  there exists $p^\star_2=p^\star_2(\a)>1$ such that for any $p>p^\star_2$ we have
	\begin{equation}\label{eq:morse-index-2}
	m(u_p^2)  =2\left\lceil\frac {2+\a}2\kappa\right\rceil +  2\left\lceil\frac{\alpha}{2}\right\rceil 
	\end{equation}
	Otherwise, if $\a=\a_n$ for some integer $n$, then there exists $p^\star_ 2>1$ such that for any $p>p^\star_2$ we have
	\begin{equation}\label{eq:morse-index-2-an}
	(2+\a) \kappa  +  2\left\lceil\frac{\alpha}{2}\right\rceil\le m(u_p^2)\le  (2+\a) \kappa+  2\left\lceil\frac{\alpha}{2}\right\rceil+ 2
	\end{equation}}
		\end{corollary}
		Here the exponents $p^\star$ and $p^\star_2$ are the same of Theorem \ref{teo-morse-1} and \ref{teo-morse-2} and $\kappa$ is as defined in \eqref{kappa}.

\subsection{ Nonradial $n-$symmetric solutions} 

Next we turn to the least energy $n$-invariant solutions $ u_{p,n}^1\in H^1_{0,n}$  constructed before, and show that, at least for some values of $n$, they do not coincide with the radial positive solution $u_p^1$. In this case it  is interesting to understand if the least energy solutions $u_{p,n}^1$ and $u_{p,m}^1$ coincide or not. Luckily this last issue has been tackled by Gladiali in \cite{Gladiali-19} who showed the following: 
\begin{proposition}[\cite{Gladiali-19} Theorem 1.1 and Theorem 1.2]\label{prop-gladiali}
Let $u\in H^1_{0,n}$ be a solution to \eqref{H}, with $p\geq 2$ if $u$ changes sign,
that satisfies  
\[m_n(u)\leq 2.\]
Then, or $u$ is radial or, else there exists a direction $e$ such that $u$ is symmetric with respect to this direction in a sector of angle $\frac {2\pi}n$ and it is strictly monotone in the angular variable in a sector of amplitude $\frac \pi n$.
\end{proposition}
From this it follows that whenever $u_{p,n}^1$ and $u_{p,m}^1$  are nonradial then they do not coincide, due to the strict angular monotonicity. 

With the aid of this last consideration, and using Corollary \ref{cor:morse-sim}, we  can prove Theorem \ref{teo:existence-1}.
	\begin{proof}[Proof of Theorem \ref{teo:existence-1}.]
		We consider the functions $u_{p,n}^1$ obtained minimizing $\mathcal E(u)$ on $\mathcal N_n$ for $n=1, 2, \dots,  \lceil \frac \a2\rceil $. First we want to show that they are not radial so that the positive solution $u_{p,n}^1$ does not coincide with $u_p^1$. To prove this we recall that   
		 $m_n(u_{p,n}^1)=1$ by \eqref{morse-index-n-symmetric}. On the other hand by Corollary \ref{cor:morse-sim} we know that $m_n(u_p^1)=1+2\left[\frac 1n \left\lceil\frac \a 2\right\rceil\right]$
		for $p>p^\star$, so that $m_n(u_p^1)>1$ when $\frac 1n\left\lceil\frac \a 2\right\rceil  \ge 1$
		meaning $n \le  \left\lceil \frac \a 2\right\rceil$.
		Then $u_p^1$ does not coincide with $u_{p,n}^1$ for any $p>p^\star$ and $n\in \{1, \dots, \lceil \frac \a2\rceil \}$ by Morse index considerations.
			It lasts to prove that $u_{p,1}^1\neq u_{p,2}^1\neq \dots \neq u_{p,\lceil \frac \a2\rceil  }^1$. But this follows from Proposition \ref{prop-gladiali}
		by the strict angular monotonicity of $u_{p,n}^1$ in a sector of amplitude $\frac {\pi}n$.
	\end{proof}

As an easy consequence, adding the radial solution $u_p^1$ then we obtain
\begin{corollary} 
Let $\a >  0$ be fixed. Then, there exists an exponent $p^\star=p^\star(\a)$ such that problem \eqref{H} admits at least $1+ \lceil \frac \a2\rceil  $ distinct positive solutions
for every $p\in (p^\star(\a),\infty)$.
\end{corollary}

As noticed in \cite{Gladiali-19} the solutions $u_{p,n}^1$ of Theorem \ref{teo:existence-1} exhibit the same monotonicity and symmetry properties of the functions $\sin n\theta$, $\cos n \theta$. Up to a rotation, they are symmetric with respect any direction $(\cos\frac{h\pi}{n},\sin\frac{h\pi}{n})$ for $h=1,\dots n$, their maxima and minima 
 either are attained  alternately for $\theta=\frac{h\pi}{n}$  and they are strictly monotone in each sector among two consecutive critical points, or the maximum (minimum) is placed in the origin and the minima (maxima) are attained at $\theta=\frac{2h\pi}{n}$. This result is consistent with some previous existence results by \cite{EPW} where positive solutions to \eqref{H}  with $n$ symmetric concentration points placed along the vertex of a regular polygon are constructed via a finite dimensional reduction method.  It is reasonable to conjecture that our solutions $u_{p,n}^1$ coincide with theirs, and this conjecture is strengthen by the fact that we obtain exactly the same number of different solutions.

A consequence of the method we used in the proof is that when $n > \left\lceil\frac \a 2\right\rceil$ then $m_n(u_p^1)=1$ and we strongly believe that $u^1_{p,n}=u_p^1$ in this case, meaning that for $\a$ fixed we can construct only a finite number, $1+\lceil \frac \a2\rceil$, of positive distinct solutions to \eqref{H}. Of course 
the number of positive solutions we found increases with $\a$,  corresponding to its even values, but for any $\a>0$ fixed the number of distinct solutions is finite and depends on $\a$. 
Precisely  for every $n =1, \dots \lceil \frac \a2\rceil $ there exists an exponent $\bar p_n=\bar p_n(\a)$ (characterized by the condition $n\, \nu_1(p) < -  4/(2+\a)^2$ as $p> \bar p_n$) such that the least energy solution $u_{p,n}^1$ is nonradial as $p>\bar p_n$. In particular since $H^1_{0,1}$ coincide with $H^1_0(B)$ then the least energy solution of \eqref{H} is nonradial when $p>\bar p_1$.
\\
Finally we recall that following \cite{AG14} it can be proved that for $p$ sufficiently close to $1$ problem \eqref{H} possesses a unique positive solution. 
The Morse index of $u_p^1$ as $p\to 1$ has been studied in \cite{Amadori} where it has been shown that there exists a $\delta>0$ such that $m(u_p^1)=1$ for $p\in (1,1+\delta)$ and that $\lceil \frac \a2\rceil$ branches of nonradial solutions bifurcate from the curve of radial solutions $p\mapsto u^1_p$  as $p\in(1,\infty)$. 
Such branches are made up by functions in $H^1_{0,n}$ and detach exactly at $\bar p_n$, it is therefore natural to conjecture that they coincide with the curve $(p, u^1_{p,n})$ as $p> \bar p_n$.


\

Next we turn to the case of nodal solutions.

\begin{proof}[Proof of Theorem \ref{teo:existence-2}.] 
	We consider the functions $u_{p,n}^2$ obtained minimizing $\mathcal E(u)$ on $\mathcal N_{n, \nod}$ for $n=1, 2, \dots, \lceil \frac {2+\a}2 \kappa-1\rceil $. First we show that they are not radial so that the solutions $ u_{p,n}^2$ do not coincide with the least-energy nodal radial solution $u_p^2$. To prove this we recall that by previous considerations $m_n(u_{p,n}^2)= 2$	for every $p>1$.
On the other hand by Corollary \ref{cor:morse-sim} we know that 
\[m_n(u_p^2) \ge  2+ 2\left[ \frac 1n  \left\lceil\frac{2+\alpha}{2}\kappa-1\right\rceil  \right]+2\left[ \frac 1n \left\lceil\frac{\alpha}{2}\right\rceil\right]\] 
for $p>p^\star_2$ where $\kappa$ is as defined in \eqref{kappa}, so that 
	$m_n(u_p^2)>2$ if $\frac 1n \left\lceil\frac{\alpha}{2}\right\rceil\geq1$ or $\frac 1n\left\lceil\frac{2+\alpha}{2}\kappa-1\right\rceil \geq 1$. But from \eqref{kappa} it suffices that $\frac 1n\left\lceil\frac{2+\alpha}{2}\kappa-1\right\rceil \geq 1$ meaning that $n\le \left\lceil\frac{2+\alpha}{2}\kappa-1\right\rceil $.
	\\
	Then $u_p^2$ does not coincide with $u^2_{p,n}$ for any $p>p^\star_2$ and $n\in \{1, \dots, \lceil \frac {2+\a}2\kappa-1 \rceil \}$,
	 and Proposition \ref{prop-gladiali} yields that,  for $p\geq 2$ every $u_{p,n}^2$ is strictly increasing w.r.t.~the angular variable in a sector of amplitude $\frac {\pi}n$, and strictly decreasing in the subsequent sector of amplitude $\frac \pi n$.
	In particular $u_{p,h}^2\neq u_{p,n}^2$ as $h\neq n \in \{1, \dots, \lceil \frac {2+\a}2\kappa-1 \rceil \}$  and $u_{p,n}\neq u_p^2$ for every $p>p_2^*$ which concludes the proof. 
\end{proof}

As an easy consequence, adding the radial solution $u_p^2$, then we obtain
\begin{corollary}
Let $\a\geq 0$ be fixed. Then, there exists an exponent  $p_2^*=p_2^*(\a)$ such that problem \eqref{H} admits at least $ \lceil \frac {2+\a}2\kappa \rceil  $ distinct positive solutions
for every $p\in (p^*_2(\a),\infty)$.
\end{corollary}

Also the solutions $u_{p,n}^2$ of Theorem \ref{teo:existence-2} exhibit the same monotonicity and symmetry properties of the functions $\sin n\theta$, $\cos n \theta$, which are
consistent with some previous existence results by \cite{ZY} where nodal solutions to \eqref{H}  with  $n$ symmetric concentration points placed along the vertex of a regular polygon are constructed via a finite dimensional reduction method, for $p$ large.  See also the paper \cite{EMP} for the Lane-Emden case corresponding to $\a=0$. \\
 We do not know if the solutions found with our construction coincide or not with the ones in \cite{ZY}, even if they possess the same type of symmetries. Observe, in fact, that when $\a=0$ the solutions in \cite{GI} are not of the type of the ones in \cite{EMP} since the nodal line of the first ones does not touch the boundary while the nodal line of the second does. We suspect that also in the H\'enon regime the same behaviour is possible, namely some of our solutions should have nodal line touching the boundary while some other not, depending on $\a$ and $n$. It should be very interesting indeed to study the asymptotic behaviour of $n$-invariant least energy solutions for large values of $p$ in order to understand this point.
\\
Moreover also in this case the number of different solutions we can find increases with $\a$, but differently form the previous case the number changes corresponding to the integer values of   $\frac {2+\a}2\kappa$ instead of the even values of $\a$. This seems to be a new phenomenon which has never been highlighted before. \\

Finally the Morse index of $u_p^2$ has been studied also in \cite{Amadori} where it has been shown that there exists $\delta>0$ such that $m(u^2_p) = 2 \left\lceil\frac{2+\a}{2}\beta\right\rceil$ for $p\in (1, 1+\delta)$. Here $\beta\approx 2,\!305$ is another fixed number. 
	In the same paper $\left\lceil\frac{2+\a}{2}\kappa-1\right\rceil -\left[\frac{2+\a}{2}\beta\right]$ branches of nonradial solutions that bifurcate from the curve of radial solutions are produced, comparing the value of the Morse index (or, better of the eigenvalue $\nu_1(p)$) at the ends of the existence range. The number of the solutions obtained by bifurcation is lower than the one obtained here by minimization on rotationally invariant spaces. 
It seems that for $n=1, \dots \left\lceil\frac{2+\a}{2}\beta-1\right\rceil$ the solutions $u_{p,n}^2$ are nonradial for every $p>1$ and the two curves $p\mapsto  u_{p,n}^2$ and $p\mapsto u^2_{p}$ do not intersect each other. This is certainly true for $n=1$ because $m( u_{p,1}^2)=2$ by \cite{BW}, while $m(u^2_p)\ge 4$ by \cite[Theorem 1.1]{AG-sez2}.  Conversely for $n=\left\lceil\frac{2+\a}{2}\beta\right\rceil, \dots \left\lceil \frac{2+\a}{2}\kappa-1\right\rceil$ the curve  $p\mapsto  u_{p,n}^2$ would coincide with the one of radial solutions for $p$ under a certain value $p_n$, and then it would bifurcate becoming nonradial.

\section{Appendix}

We report here the almost straightforward  proofs of some basic facts. First we show some characterizations of the weighted Sobolev spaces that we have used  in Section \ref{sec:3}, next we study the limit singular eigenvalue problems \eqref{eq:finale}.

\subsection{Some density properties in weighted Sobolev spaces}

Throughout the paper we have used the following notations for some weighted Lebesgue spaces:
\begin{align*}
{\mathcal L}:=& \big\{w:B\to \R \text{ measurable and s.t. } \int_B |x|^{-2} w^2 dx < +\infty\big \}, \\
{\mathcal L}(\R^2):=& \big\{w:\R^2\to \R \text{ measurable and s.t. } \int_{\R^2} |x|^{-2} w^2 dx  < +\infty\big \}
\end{align*}
and $\mathcal L_{\rad}$,  $\mathcal L_{\rad}(\R^2)$ as the subspaces of $\mathcal L(B)$ and of ${\mathcal L}$ made up by radial functions.
Starting from this we have introduced 
\[ \mathcal H_0:=H^1_0(B)\cap {\mathcal L}, \quad \mathcal H_{0,\rad} := H^1_0(B)\cap {\mathcal L}_{\rad},  \quad 
\mathcal D_{\rad}:= D^{1,2}(\R^2)\cap {\mathcal L}_{\rad}(\R^2), \]
where  $H^1_0(B)$ and $D^{1,2}(\R^2)$ are the usual spaces, namely 
 $H^1_0(B)$ is the closure of $C^{\infty}_0(B)$ under the norm $\left(\int_B |\nabla w|^2 dx \right)^{\frac{1}2}$ and $D^{1,2}(\R^2)$  is the closure of $C^{\infty}_0(\R^2)$ under the norm $\left(\int_{\R^2}  |\nabla w|^2 dx \right)^{\frac{1}2}$. 
Such weighted Sobolev spaces have been endowed with the natural norms
\begin{align*}
\| w \|_{\mathcal H_0} = & \left(\int_B \left( |\nabla w|^2 + |x|^{-2} w^2\right) dx \right)^{\frac{1}2}, \\
\| g\|_{\mathcal D_{\rad}}= & \left(\int_0^{\infty} \left( r |g'|^2 + r^{-1} g^2\right) dr \right)^{\frac{1}2}.
\end{align*}
We check that, using these norms, 	$C^{\infty}_0(B\setminus\{0\})$ is dense in $\mathcal H_0$ and similarly  $C^{\infty}_0(0,\infty)$ is dense in $\mathcal D_{\rad}$.
\begin{lemma}\label{prop-density}
	$\mathcal H_0$ is the closure of 	$C^{\infty}_0(B\setminus\{0\})$ with respect to the norm  $\| \cdot \|_{\mathcal H_0}$.
\end{lemma}
\begin{proof}
	Let $\phi\in \mathcal H_0$ and  $\e>0$. We can choose $0<\rho<1/2$ so small that
	\begin{equation}\label{int-piccolo}
	\int_{B_{2\rho}}|\nabla \phi|^2+|x|^{-2}|\phi|^2 \ dx<\frac {\e^2}{33}
	\end{equation}
	Let $\Psi\in C^1(B)$ be a cut-off function with the properties
	\begin{equation}\label{eq:cut-off-2} \begin{split}
	0\le \Psi(x) \le 1 , \quad 
	\Psi(x) =\begin{cases}
	0 & \text{ as }  \ 0\le |x| \le \rho,  \\
	1 & \text{ as } \ 2\rho\leq |x|< 1 ,
	\end{cases}   \quad 
	\left| \nabla \Psi(x)\right|  \leq  
	\frac 2{\rho} . 
	\end{split} \end{equation} 
	First we show that
	\begin{equation}\label{prima-norma-piccola}
	\nor \phi(1-\Psi)\nor_{\mathcal H_0}<\e.
	\end{equation}
		Indeed using that, 	by the choice of $\psi$, 
		\[ |\nabla \left(\phi(1-\Psi)\right)|^2
		\leq 2(1-\Psi)^2|\nabla \phi|^2+2\phi^2|\nabla (1-\Psi)|^2 \le 2 |\nabla \phi|^2 + \frac{8}{\rho^2}\phi^2, \]
gives
	\begin{equation}\begin{split}
	\int_B |\nabla \left(\phi(1-\Psi)\right)|^2 dx  
	=\int_{B_{\rho}}|\nabla \phi|^2dx +\int_{B_{2\rho}\setminus B_{\rho}}|\nabla \left(\phi(1-\Psi)\right)|^2dx \\
	\le \int_{B_{\rho}}|\nabla \phi|^2dx+
	2 \int_{B_{2\rho}\setminus B_{\rho}}|\nabla \phi|^2 dx+\frac{8}{\rho^2} \int_{B_{2\rho}\setminus B_{\rho}}\phi^2 dx \\
	\le  2\int_{B_{2\rho}}|\nabla \phi|^2+ 32 \int_{B_{2\rho}\setminus B_{\rho}}|x|^{-2}\phi^2< \frac{32\e^2}{33}\label{eq:minore-di-epsilon}
	\end{split}
	\end{equation}
	where last step follows by  \eqref{int-piccolo}.
	Furthermore	\[\int_B |x|^{-2} \phi^2 (1-\Psi)^2 dx \le \int_{B_{2\rho}} |x|^{-2} \phi^2 dx <\frac {\e^2}{33}
		\]
		thanks to \eqref{int-piccolo}. 
	\\
	Next we show that there is a sequence $\phi_n\in C^\infty_0(B\setminus\{0\})$ with support contained in $B\setminus B_{\rho}$ such that 
	\begin{equation}\label{conv-phi-n}
	\phi_n\to \phi\Psi \ \ \text{ in }\mathcal H_0
	\end{equation}
	as $n\to \infty$. \eqref{conv-phi-n} together with \eqref{eq:minore-di-epsilon} implies that
	\[\nor \phi_n-\phi\nor_{\mathcal H_0}\leq \nor \phi_n-\phi\Psi\nor_{\mathcal H_0}+\nor \phi\Psi-\phi\nor _{\mathcal H_0}<2\e\]
	for $n$ large enough, showing the density of $C^\infty_0(B\setminus\{0\})$ into $\mathcal H_0$ and
	concluding the proof.\\
	We turn then to the proof of \eqref{conv-phi-n}. Indeed $\phi\Psi\in H^1_0(B\setminus B_{\rho})  $, and, since $C^\infty_0(B\setminus B_{\rho})$ is dense in $H^1_0(B\setminus B_{\rho})  $ there exists a sequence $\phi_n \in C^\infty_0(B\setminus B_{\rho})$ such that $\phi_n\to \phi\Psi$ in $H^1_0(B\setminus B_{\rho})$.
	Next we extend $\psi_n$ to be zero in $B_{\rho}$ so that $\phi_n\in C^\infty_0(B\setminus \{0\})$ and satisfies
	\begin{equation}\label{conv-grad}\int_B|\nabla (\phi_n-\phi\Psi)|^2 dx =
	\int_{B\setminus B_{\rho}}|\nabla (\phi_n-\phi\Psi)|^2 dx \to 0\end{equation}
	and clearly by the Poincar\'e inequality in $H^1_0(B\setminus B_{\rho})$
	\[\int_B |\phi_n-\phi\Psi|^2 dx =\int_{B\setminus B_{\rho}}|\phi_n-\phi\Psi|^2 dx \to 0\]
	as $n\to \infty$.
	From this last equality follows also that
	\[\begin{split}
	\int _B |x|^{-2}|\phi_n-\phi\Psi|^2dx &=\int_{B\setminus B_{\rho}}|x|^{-2}|\phi_n-\phi\Psi|^2 dx 
	\le \frac 1{\rho^2} \int_{B\setminus B_{\rho}}|\phi_n-\phi\Psi|^2 dx \to 0
	\end{split}\]
	and together with \eqref{conv-grad} proves \eqref{conv-phi-n}.
\end{proof}

	\begin{lemma}\label{lem-Drad} 
	$\mathcal D_{\rad}$ is the closure of 	$C^{\infty}_{0}(0,\infty)$ with respect to the norm $\| \cdot \|_{\mathcal D_{\rad}}$. 
	\end{lemma}
	\begin{proof}
	We take $\phi\in  D_{\rad}$ and  $\e>0$ and show that there exists $\psi \in C^{\infty}_{0}(0,\infty)$
		such that
		\begin{equation}\label{denso}
		\| \phi - \psi\|_{\mathcal D_{\rad}} < 2\e.
		\end{equation}
		To begin with, we  choose $0<\rho<1/2$ so small that
		\begin{equation}\label{int-piccolo-d}
		\int_0^{2\rho} \left(  r |\phi'|^2+r^{-1} \phi^2 \right) dr + \int_{\frac{1}{\rho}}^{\infty} \left( r |\phi'|^2 + r^{-1} \phi^2 \right) dr < \frac{\e^2}{33},
		\end{equation}
		and we take  a cut-off function $\Psi\in C^1_{\rad}(\R^2)$  with the properties
		\begin{equation}\label{eq:cut-off-2-d} \begin{split}
		0\le \Psi(r) \le 1 , \quad & 
		\Psi(r)=\begin{cases}
		0 & \text{ as }  \ 0\le r < \rho \ \text{ and } \ 2/\rho < r , \\
		1 & \text{ as } \ 2\rho\leq r \le  {1}/{\rho} ,
		\end{cases}   \\
		& \left| \Psi'(r)\right|\leq  \begin{cases}
		2/{\rho}   & \text{ as }  \rho \le r \le 2\rho , \\
		2\rho  & \text{ as } \  {1}/{\rho} \le 2/{\rho}.
		\end{cases}
		\end{split} \end{equation} 
		Clearly $\| \phi - \psi\|_{\mathcal D_{\rad}} \le 	\| \phi (1- \Psi)\|_{\mathcal D_{\rad}}  + \| \phi \Psi - \psi\|_{\mathcal D_{\rad}}$
		and \eqref{denso} follows after checking separately that
		\begin{align}\label{denso-2}
		\nor \phi(1-\Psi)\nor_{\mathcal D_{\rad}}<\e , 
		\intertext{and that there exists $\psi \in C^{\infty}_{0}(0,\infty)$ such that}
		\label{denso-3}
		\nor \phi\Psi - \psi \nor_{\mathcal D_{\rad}}<\e .
		\end{align}
		Concerning \eqref{denso-2} we have
		\begin{align*}\nonumber 
		\int_0^{\infty} r \,  |\big(\phi(1-\Psi)\big)'|^2 dr  \le  2 	\int_0^{\infty} r |\phi'|^2 (1-\Psi)^2 dr + 2 	\int_0^{\infty} r \phi^2|(1-\Psi)'|^2 dr 
		\intertext{and \eqref{eq:cut-off-2-d} gives} \nonumber 
		\le  2 \int_{0}^{2\rho}  r |\phi'|^2 dr +  2\int_{\frac{1}{\rho}}^{\infty} r |\phi'|^2  dr 
		+ 8\int_{\rho}^{2\rho}  \frac{r}{\rho^2}  \phi^2 dr + 8 \int_{\frac{1}{\rho}}^{\frac{2}\rho}  r\rho^2  \phi^2 dr \le \\
		2  \int_{0}^{2\rho}  r |\phi'|^2 dr + 2\int_{\frac{1}{\rho}}^{\infty} r |\phi'|^2  dr  +32 \int_{\rho}^{2\rho}  r^{-1} \phi^2 dr +  32\int_{\frac{1}{\rho}}^{\frac{2}\rho}  r^{-1} \phi^2 dr 		 < \frac{32\e^2}{33}
		\end{align*}
		by \eqref{int-piccolo-d}.
	Besides 
	\[\int_0^{\infty} r^{-1} |\phi(1-\Psi)|^2 dr \le \int_0^{2\rho} r^{-1} \phi^2 dr + \int_{\frac{1}{\rho}}^{\infty}  r^{-1} \phi^2 dr < \frac{\e^2}{33}
	\]
	by \eqref{eq:cut-off-2-d} and \eqref{int-piccolo-d}.	
	\\
	Turning to \eqref{denso-3}, it suffices to see that there exists a sequence $\psi_n\in C^{\infty}_{0}(0,\infty)$ with support contained in $(\rho , 2/\rho)$ such that 
		\[
		\nor \phi\Psi - \psi_n \nor_{\mathcal D_{\rad}}^2 =  	\int_{\rho}^{\frac{2}{\rho}} r \left| (\phi\Psi)'- \psi_n'\right|^2 dr  + \int_{\rho}^{\frac{2}{\rho}} r^{-1} \left| \phi\Psi'- \psi_n\right|^2 dr  \to 0
		\]
		as $n\to\infty$.
		But 	since $\phi\Psi \in H^1_{0,\rad}(B_{\frac{2}{\rho}}\setminus B_{\rho})$  it is clear that there is a sequence in $C^{\infty}_{0}(\rho, 2/\rho)$ such that 
		$\nor \phi\Psi - \psi_n\nor_{H^1_0(B_{\frac{2}{\rho}}\setminus B_{\rho})}\to 0$. Extending $\psi_n$ to zero on $(0,\rho)$ and $(\frac{2}{\rho},\infty)$ gives the needed sequence because clearly $\displaystyle \int_{\rho}^{\frac{2}{\rho}}  r \left| (\phi\Psi)'- \psi_n'\right|^2 dr \to 0$, but also 
		\begin{align*} \nonumber
	\int_{\rho}^{\frac{2}{\rho}} r^{-1} \left|\phi\Psi- \psi_n \right|^2 dr 
		\le \frac{1}{\rho^2} \int_{\rho}^{\frac{2}{\rho}} r \left|\phi\Psi- \psi_n \right|^2 dr 
		\intertext{ and using Poincar\'e inequality in $H^1_{0}(B_{\frac{1}{\rho}}\setminus B_{\frac{\rho}{2}})$  gives}
		\le  \frac{C}{\rho^2}  \int_{\rho}^{\frac{2}{\rho}}  r \left| (\phi\Psi)'- \psi_n'\right|^2 dr \to 0 .
		\end{align*}
		\end{proof}

\subsection{The limit eigenvalue problems}

	Here we describe the negative eigenvalues and the respective eigenfunctions of the two limit eigenvalue problems
	\[\tag{\ref{eq:finale}}
	\begin{cases} -\left(r \eta'\right)'= r \left(W^i+\frac{\beta^i }{r^2} \right) \eta \quad &   r\in(0,\infty), \\
	\eta\in {\mathcal D}_{\rad}   &\end{cases}
	\]
	 where $\mathcal D_{\rad}$ is as defined in \eqref{def:D},
	\begin{align*}
	W^1(r) = \frac{64}{(8+r^2)^2} ,  \qquad W^2(r)= \frac{8 \kappa^2\delta  r^{2\kappa-2}}{    \left(\delta +r^{2\kappa}\right)^2} 
	\end{align*}
	with  $\kappa= \frac{2+\gamma}{2}$ as in \eqref{kappa} and $\gamma$ and $\delta$ have been fixed in \eqref{eq:relazioni-gamma-delta-l}. 

Let
\begin{equation}\label{def-primo-autov-limite}
	\beta^i_1 = \inf\limits_{\psi\in \mathcal D_{\rad}\setminus\{0\}} \frac{\int_0^{\infty} r \left( |\psi'|^2 - W^i \psi^2\right) dr }{\int_0^{\infty} r^{-1}  \psi^2 dr } 
	\end{equation}

	If $\beta^i_1<0$,  since the functions $W^i$ decay at infinity the arguments of \cite[Proposition 3.1]{AG-sez2} can be adapted to see that 
	it is attained by a function $\eta^i_1\in {\mathcal D}_{\rad}$
	which solves \eqref{eq:finale} for $\beta^i=\beta^i_1$ in weak sense, i.e.
	\[\tag{\ref{eq:finale-sol}}
	\int_0 ^{\infty} r (\eta^i)' \varphi' \, dr = \int_0^{\infty} r \left( W^i +\frac{\beta^i }{r^2} \right) \eta^i \varphi \, dr 
	\]
	for every $\varphi \in {\mathcal D}_{\rad}$ or equivalently for every $\varphi\in C^{\infty}_0(0,\infty)$.
	In this case the function $\eta^i_1$ is called an eigenfunction related to $\beta^i_1$. 
	 \begin{lemma}\label{attained}
			If $\beta^i_1<0$, then it is attained by a function $\eta^i_1\in {\mathcal D}_{\rad}$
			which solves \eqref{eq:finale} for $\beta^i=\beta^i_1$ in weak sense.
		\end{lemma}
		\begin{proof}
			Let us consider a minimizing sequence $\psi_n\in \mathcal D_{\rad}$
			with
			\begin{equation}\label{minimizing}
			\int_0^{\infty} r \left( |\psi_n'|^2 - W^i \psi_n^2\right) dr = \beta_n \int_0^{\infty} r^{-1} \psi_n^2 dr , \quad \beta_n \to \beta^i_1<0 .
			\end{equation}
			Without loss of generality we may take that $\psi_n$ is normalized such that 
			\begin{equation}\label{norm}
			\int_0^{\infty} r^{-1} \psi_n^2 dr =1   \end{equation}
			Hence $\psi_n$ is bounded in $\mathcal D_{\rad}$ because
			\[\int_0^{\infty} r |\psi_n'|^2 dr \le \int_0^{\infty} r W^i \psi_n^2 dr \le  \sup\limits_{[0,\infty)} \left|r^2 W^i (r)\right|  \int_0^{\infty} r^{-1} \psi_n^2 dr  \le C \]
			thanks to \eqref{norm} and the fact that $r^2 W^i(r) \to 0$ as $r\to\infty$.
			So,  up to a subsequence, $\psi_n \to \eta\in \mathcal D_{\rad}$ weakly in $\mathcal D_{\rad}$ and strongly in $L^2(B_R)$  for every $R>0$.
			Let us check that 
			\begin{equation} \label{pippo} 
			\int_0^{\infty} r W^i \psi_n^2 dr \to  \int_0^{\infty} r W^i \eta^2 dr .
			\end{equation}
			Indeed for every $\e>0$, there exists $R>0$ such that
			\begin{equation}\label{scelta-epsilon}
			\sup\limits_{[R,\infty)} \left|r^2 W^i (r)\right|  < \e .
			\end{equation}
			Next
			\begin{align*} \left| \int_0^{\infty} r W^i \left(\psi_n^2 - \eta^2 \right) dr \right| \le 
			\left|\int_0^R  r W^i \left(\psi_n^2 - \eta^2 \right) dr \right|+ \left| \int_R^{\infty} r W^i \left(\psi_n^2 - \eta^2 \right) dr \right|  .
			\end{align*}
			But
			\begin{align*}
			\left|\int_0^R  r W^i \left(\psi_n^2 - \eta^2 \right) dr \right|   \le\sup_{(0,R)}W^i\int _0^R r |\psi_n^2 -\eta^2| dr
			\le C \int _0^R r |\psi_n^2 -\eta^2| dr \to 0
			\end{align*}
			as $n\to \infty$ because  $\psi_n\to \eta$ in  $L^2(B_R)$, and 
			\begin{align*}
			\left|\int_R^{\infty}  r W^i \left(\psi_n^2 - \eta^2 \right) dr \right| \le 
			\sup\limits_{[R,\infty)} \left|r^2 W^i (r)\right| \left( \int_0^{\infty} r^{-1} \psi_n^2 dr +  \int_0^{\infty} r^{-1}\eta^2 dr \right)
			\le C \e 
			\end{align*}
			by \eqref{scelta-epsilon}
			and since the Fatou's Lemma implies that
			\[ \int_0^{\infty} r^{-1}\eta^2 dr\le \liminf_{n\to \infty}\int_0^{\infty} r^{-1} \psi_n^2 dr=1 .
			\]
			
			Next we check that $\eta$ minimizes the quotient in \eqref{def-primo-autov-limite}, namely 
			\[\int_0^{\infty} r |\eta'|^2 dr -\int_0^{\infty} r W^i\eta^2 dr - \beta^i_1\int_0^{\infty} r^{-1} \eta^2 dr \leq 0 .
			\]
			Since $\beta^i_1<0$, Fatou's Lemma applies giving that 
			\begin{align*} 
			\int_0^{\infty} r |\eta'|^2 dr -\int_0^{\infty} r W^i\eta^2 dr - \beta^i_1\int_0^{\infty} r^{-1} \eta^2 dr  \\ \leq
			\liminf_{n\to \infty} \Big(\int_0^{\infty} r |\psi_n'|^2 dr - \beta_n \int_0^{\infty} r^{-1} \psi_n^2 dr\Big)- \int_0^{\infty} r W^i\eta^2 dr \\
			\underset{\eqref{minimizing}}{=} \liminf_{n\to \infty} \int_0^{\infty} r W^i\psi_n^2 dr - \int_0^{\infty} r W^i\eta^2 dr = 0
			\end{align*}
			by \eqref{pippo}. 
			Eventually, as $\eta$ minimizes  the quotient in \eqref{def-primo-autov-limite}, it is standard to see that  \eqref{eq:finale-sol} holds.
		\end{proof}
	}
	
		Next, if 
	\[
	\beta^i_2:=\inf\limits_{\stackrel{\psi\in \mathcal D_{\rad}\setminus\{0\}}{\int_0^{\infty} r^{-1} \eta^i_1 \psi dr =0}} \frac{\int_0^{\infty} r \left( |\psi'|^2 - W^i \psi^2\right) dr }{\int_0^{\infty} r^{-1}  \psi^2 dr } < 0,
	\]
	the same arguments of Lemma \ref{attained} yield that it is attained by an eigenfunction $\eta^i_2\in \mathcal D_{\rad}$ which solves \eqref{eq:finale} for $\beta^i=\beta^i_2$ in weak sense.   Further such $\eta^i_2$ is the weak limit in ${\mathcal D}_{\rad}$ of a minimizing sequence $\eta^i_{2,n}$ which satisfies $	\int_0^{\infty} r^{-1} \eta^i_1 \eta^i_{2,n} dr =0$, and so 
	\begin{equation}\label{ortog-eigenf}
	\int_0^{\infty} r^{-1} \eta^i_1 \eta^i_2 dr =0.
	\end{equation}
	 Conversely, one can see that if \eqref{eq:finale} has a nontrivial solution $\eta^i$ corresponding to some $\beta^i<0$, then such $\beta^i$ is an eigenvalue according to \eqref{def-primo-autov-limite} and $\eta^i$ is the related eigenfunction.
	
	We prove that
	
	\begin{proposition}\label{prop:beta_1}
		As $i=1$, $\beta^1_1=-1$ is the only negative eigenvalue of \eqref{eq:finale}. It is simple and the related eigenfunction is
		\begin{equation*}
		\eta^1_1(r):=\frac{4r}{8+r^2} .
\tag{\ref{prima-autof-limite-1}}		\end{equation*}
	\end{proposition}
	\begin{proposition}\label{prop:beta_2}
		As $i=2$, $\beta^2_1=-\kappa^2$ is the only negative eigenvalue of \eqref{eq:finale}. It is simple and the related eigenfunction is
		\begin{equation}\label{prima-autof-limite-2}
		\eta^2_1(r):=\frac{r^{\kappa}}{\delta+r^{2\kappa}} .
		\end{equation}
	\end{proposition}
	
	\begin{proof}[Proof of Proposition \ref{prop:beta_1}]
		In \cite[Sec 5]{DIPN=2} it has been shown that $\beta^1_1=-1$ is the first eigenvalue of \eqref{eq:finale} and $\eta^1_1$ given by \eqref{prima-autof-limite-1} is a related eigenfunction.
		It remains to show that
		\[
		\beta^1_2:=\inf\limits_{\stackrel{\psi\in \mathcal D_{\rad}\setminus\{0\}}{\int_0^{\infty} r^{-1} \eta^1_1 \psi dr =0}} \frac{\int_0^{\infty} r \left( |\psi'|^2 - W^1 \psi^2\right) dr }{\int_0^{\infty} r^{-1}  \psi^2 dr } \ge 0 .
		\]
		Assume by contradiction that $\beta^1_2<0$, then there exists $\eta^1_2 \in \mathcal D_{\rad}$ which solves \eqref{eq:finale} for $\beta^1=\beta^1_2$ in weak sense, and satisfies
		\eqref{ortog-eigenf}. 
		By \eqref{ortog-eigenf} it follows that there exists $R>0$ such that $\eta^1_2(R)=0$, and $\eta^1_2$ is not constantly zero on $(0,R)$ either on $(R,\infty)$. So the functions  
		\[
		\psi_1(r) = \begin{cases} \eta^1_2(r) & \ \text{ as } 0\le r < R \\ 0 & \ \text{ as } r \ge R \end{cases} 
		\ \text{ and } \ 
		\psi_2(r) = \begin{cases} 0 & \ \text{ as } 0\le r \le R \\ \eta^1_2(r) & \ \text{ as } r \ge R .\end{cases} 
		\]
		are not trivial and belong to ${\mathcal D}_{\rad}$. Using them as test functions in the weak formulation \eqref{eq:finale-sol} gives
		\begin{align}
		\label{d-contr} \int_0^{\infty}  r \left( (\psi'_1)^2 - W^1  (\psi_ 1)^2 \right)dr = \int_0^{R} r \left( (\psi'_1)^2 - W^1(\psi_ 1)^2 \right)dr & = \beta^1_ 2 \int_0^{R} r^{-1} (\psi_1)^2 dr < 0   , \\
		\label{f-contr} \int_0^{\infty}  r \left( (\psi'_2)^2 - W^1(\psi_ 2)^2 \right)dr = \int_{R}^{\infty} r \left( (\psi'_2)^2 - W^1(\psi_ 2)^2 \right)dr & = \beta^1_ 2 \int_{R}^{\infty} r^{-1} (\psi_2)^2 dr < 0  .
		\end{align}
		Next we compare $\psi_1$ and $\psi_2$ with the function
		\[\zeta(r)= \frac{8-r^2}{8+r^2} , \quad r\ge 0\]
		which solves in classical sense 
		\begin{equation}\label{eq:zeta}  -\left(r \zeta'\right)'= r W^1\zeta \quad \qquad    r>0 \end{equation}
		and satisfies $\zeta > 0$ on $[0, \sqrt{8})$, $\zeta<0$ on $(\sqrt{8}, \infty)$, $\zeta(\sqrt{8})=0$.
		 Notice that $\zeta$ does not belong to $\mathcal D_{\rad}$, anyway its restriction to $[0,\sqrt{8}]$ belongs to
		  the space $H^1_{0,\rad}(0, \sqrt{8})$, i.e. the set of functions on  $(0, \sqrt{8})$ which have first order weak derivative satisfying
		\[ \int_0^{\sqrt{8}} r \left( (\psi')^2 + \psi^2 \right) dr < \infty \ \text{ and } \psi(\sqrt{8})=0 .\]
		Hence it is an eigenfunction for the regular eigenvalue problem
		\[
		\begin{cases} -\left(r \psi'\right)'= r \left(W^1+\sigma \right) \psi \quad &   r\in(0,\sqrt{8}), \\
		\psi \in  H^1_{0,\rad}(0, \sqrt{8})  
		\end{cases}\]
		corresponding to $\sigma=0$, and since $\zeta > 0$ on $[0, \sqrt{8})$  it must be a first eigenfunction, implying that
		\begin{equation}\label{d} \int_0^{\sqrt{8}} r \left( (\psi')^2 - W^1(\psi)^2 \right)dr  \ge 0  \ \text{ for every } \psi \in H^1_{0,\rad}(0, \sqrt{8}) .\end{equation}
		If $R\le \sqrt{8}$, then $\psi_1 \in H^1_{0,\rad}(0, \sqrt{8})$ since its support is contained in $[0,\sqrt{8}]$ and 
		\begin{align*} \int_0^{\sqrt{8}} r \left( (\psi_1')^2 + \psi_1^2 \right) dr  \le  \int_0^{\sqrt{8}} r (\psi_1')^2 dr  + 8 \int_0^{\sqrt{8}} r^{-1} \psi_1^2 dr .
		\end{align*}
		But in this case \eqref{d-contr} would contradict \eqref{d}, therefore $R>\sqrt{8}$.
		
		Next we show that  $R>\sqrt{8}$ clashes with \eqref{f-contr}, concluding the proof.
		To this aim we perform a Kelvin transform and define
		\begin{align*}
		\widehat \zeta (r) =  \zeta(r^{-1}) =  \frac{8r^2-1}{8r^2+1},  \qquad  \widehat \psi_2(r)  = \psi_2(r^{-1}) , \qquad \widehat W(r) = r^{-4}W^1(r^{-1}) = \frac{64}{(1+8r^2)^2}
		\end{align*}
		as $0\le r \le 1/\sqrt{8} $.
		Notice that $\widehat\zeta \in H^1_{0,\rad}(0,1/\sqrt{8})$ and $\widehat W(r) $ is bounded.
		It is not hard to see that $\widehat\zeta$ is an eigenfunction for
		\[
		\begin{cases} -\left(r \psi'\right)'= r \left( \widehat W+\sigma \right) \psi \quad &   r\in(0,1/\sqrt{8}), \\
		\psi \in H^1_{0,\rad}(0, 1/\sqrt{8})  
		\end{cases}\]
		corresponding to $\sigma=0$.
		Indeed 
		\[ -\left(r {\widehat\zeta}'(r)\right)' =  \left(r^{-1} {\zeta}'(r^{-1})\right)' =   -  \frac{d}{ds} \left( s {\zeta}'(s) \right)\Big|_{s=\frac{1}{r}} 
		\underset{\eqref{eq:zeta}}{=} r^{-2} \left( s W^1(s) \, \zeta(s)\right)\Big|_{s=\frac{1}{r}}   = r \widehat W \, {\widehat\zeta}(r) . \]
		Besides since $\widehat\zeta<0$ on $[0,1/\sqrt{8})$, it must be a first eigenfunction, implying that
		\begin{equation}\label{f} 
		\int_0^{1/\sqrt{8}} r \left( (\psi')^2 - \widehat W (\psi)^2 \right)dr  \ge 0  \ \text{ for every } \psi \in H^1_{0,\rad}(0, 1/\sqrt{8}) .
		\end{equation}
		On the other hand, as we are taking that $R>\sqrt{8}$, the support of $\widehat \psi_2$ is  contained in $[0, 1/\sqrt{8})$ and 
		$\widehat \psi_2\in H^1_{0,\rad}(0,1/\sqrt{8})$ since 
		\begin{align*}
		\int_0^{1/\sqrt{8}} r \left( ({\widehat\psi_2}')^2 + {\widehat\psi_2}^2\right) dr 
		= \int_{\sqrt{8}}^{\infty} \left(r (\psi_2')^2 +r^{-3} \psi_2^2 \right)dr 
		\le  \int_{\sqrt{8}}^{\infty} r (\psi_2')^2 dr + 8 \int_{\sqrt{8}}^{\infty} r^{-1} \psi_2^2 dr 
		\end{align*}
		with $\psi_2\in {\mathcal D}_{\rad}$.
		Eventually 
		\begin{align*}
		\int_0^{1/\sqrt{8}} r \left( (\widehat\psi_2')^2 - \widehat W  (\widehat\psi_2)^2 \right)dr = \int_0^{1/\sqrt{8}} r^{-3} \left(  (\psi_2'(r^{-1}))^2 - W(r^{-1}) (\psi_2(r^{-1}))^2 \right)dr  
		\\ = \int_{\sqrt{8}}^{\infty} r \left( ( \psi_2')^2 - W^1(\psi_2)^2 \right)dr \underset{R>\sqrt{8}}{=}  \int_{R}^{\infty} r \left( (\psi_2')^2 - W^1(\psi_2)^2 \right)dr <0
		\end{align*}
		by \eqref{f-contr}, which contradicts \eqref{f}.
	\end{proof}
	
	\begin{proof}[Proof of Proposition \ref{prop:beta_2}]
		It is easy to see that $\eta^2$ is an eigenfunction for \eqref{eq:finale} with $i=2$ related to some eigenvalue $\beta^2$ if and only if $\eta^2(r) = \eta^1(\sqrt{\frac{8}{\delta}} \, r^{\kappa})$ and $\eta^1$ is an eigenfunction for \eqref{eq:finale} with $i=1$ related to the eigenvalue $\beta^1=\beta^2/\kappa^2$. 
		Therefore Proposition \ref{prop:beta_1} concludes the proof.
	\end{proof}


\begin{thebibliography}{9999}

\bibitem[AG]{AdG}{\sc Adimurthi, M. Grossi}, Asymptotic estimates for a two-dimensional problem with polynomial nonlinearity, {\em Proc. Am. Math. Soc.} {\bf 132} (2004) 1013-1019

\bibitem[AP]{AP}{\sc A. Aftalion, F. Pacella},  Qualitative properties of nodal solutions of semilinear elliptic equations in radially symmetric domains, {\em C. R. Acad. Sci. Paris, Ser. I} {\bf 339}, (2004), 339-344.

\bibitem[A]{Amadori}{\sc A.L. Amadori},{ A note on the nodal solutions of the asymptotically linear H\'enon problem. In preparation}
          
\bibitem[AG1]{AG14} {\sc A.L. Amadori, F. Gladiali}, Bifurcation and symmetry breaking for the H\'enon equation, {\em Advances in Differential Equations} {\bf 19} (2014) n.7-8,  755-782, http://projecteuclid.org/euclid.ade/1399395725

\bibitem[AG2]{AG-sez2} {\sc A.L. Amadori, F. Gladiali}, On a singular eigenvalue problem and its applications in computing the Morse index of solutions to semilinear PDE's, (2018) arXiv:1805.04321

\bibitem[AG3]{AG-N3} {\sc A.L. Amadori, F. Gladiali},  Asymptotic profile and Morse index of nodal radial solutions to the H\'enon problem, (2018) arXiv:1810.11046


\bibitem[CCN]{CCN} {\sc A. Castro, J. Cossio, J.M. Neuberger}, { A sign-changing solution for a superlinear Dirichlet problem, {\em Rocky Mountain J.Math.}, {\bf 27} (1997), 1041-1053.}



\bibitem[BWe]{BW}{\sc T. Bartsch, T. Weth}, A note on additional properties of sign changing solutions to superlinear elliptic equations,
  {\em Topological Methods in Nonlinear Analysis} {\bf 22} (2003), 1-14.
  
  \bibitem[BWW]{BWW}
{\sc T. Bartsch, T. Weth, M. Willem,} Partial symmetry of least energy nodal solutions to
some variational problems, {\em J. Anal. Math. }{\bf 96} (2005), 1-18.
  
  \bibitem[DIP]{DIPN=2}  {\sc F. De Marchis, I. Ianni, F. Pacella},
	Exact Morse index computation for nodal radial solutions of Lane-Emden problems
	(2017) {\em Mathematische Annalen}, {\bf 367} (1-2), pp. 185-227. DOI: 10.1007/s00208-016-1381-6
	
	\bibitem[EMP]{EMP} {\sc P. Esposito, M. Musso, A. Pistoia},
	On the existence and profile of nodal solutions for a two-dimensional elliptic problem with large exponent in nonlinearity
	{\em Proceedings of the London Mathematical Society}, 94/2 (2007), 497-519. DOI: 10.1112/plms/pdl020 
	
	\bibitem[EPW]{EPW} {\sc P. Esposito, A. Pistoia, J. Wei}, 
	Concentrating solutions for the H\'enon equation in $\mathbb{R}^2$. DOI: 10.1007/BF02916763
	
	
		\bibitem[G]{Gladiali-19} {\sc F. Gladiali}, A monotonicity result under symmetry and Morse index constraints in the plane, (2019) preprint 


\bibitem[GGN]{GGN}  {\sc F. Gladiali, M. Grossi, S.L.N. Neves}, Nonradial solutions for the H\'enon equation in $\R^N$, {\em Advances in Mathematics}, {\bf 249} (2013), 1-36, doi:10.1016/j.aim.2013.07.022.

\bibitem[GGN16]{GGN2}{\sc F. Gladiali, M. Grossi, and S. L. N. Neves}, Symmetry breaking and Morse index of solutions of nonlinear elliptic problems in the plane {\em Commun. Contemp. Math.} {\bf 18} (2016), doi: 10.1142/S021919971550087X
 
 \bibitem[GI]{GI}{\sc F. Gladiali, I. Ianni}, Quasiradial nodal solutions for the Lane-Emden problem in the ball, (2017) arXiv:1709.03315
  
  
  \bibitem[GGP1]{GGP13} {\sc M.~Grossi, C.~Grumiau, F.~Pacella},   Lane–Emden problems: Asymptotic behavior of low energy nodal solutions, {\em Ann. I. H. Poincar\'e -AN} {\bf 30} (2013), 121-140.
  
  \bibitem[GGP2]{GGP14} {\sc M.~Grossi, C.~Grumiau, F.~Pacella}, 
Lane Emden problems with large exponents and singular Liouville equations, {\em 
  Journal des Mathematiques Pures et Appliquees} {\bf 101/6} (2014),  735-754, DOI: 10.1016/j.matpur.2013.06.011

  \bibitem[H]{H} {\sc M. H\'enon}, Numerical experiments on the stability oh spherical stellar systems. {\em Astronom. Astrophys.} {\bf 24} (1973), 229-238.

\bibitem[KW]{KW}  {\sc J. K\"ubler, T. Weth} Spectral asymptotics of radial solutions and nonradial bifurcation for the H\'enon equation, (2019) arXiv:1901.00453

\bibitem[LWZ]{LWZ} {\sc Z. Lou, T. Weth, Z. Zhang} Symmetry breaking via Morse index for equations and systems of H\'enon-Schrodinger type, preprint arXiv:1803.02712

\bibitem[N]{Ni} {\sc W. M. Ni},  A Nonlinear Dirichlet Problem on the Unit Ball and Its Applications. {\em Indiana Univ. Math. J.} {\bf 31}, (1982), 801-807.
    
\bibitem[NN]{NN} {\sc  W.M. Ni, R. D. Nussbaum}, Uniqueness and nonuniqueness for positive radial solutions of $\Delta u+f(u,r)=0$, {\em Comm. Pure Appl. Math.} {\bf 38} (1985), 67-108.


\bibitem[PW]{PacellaWeth} {\sc F. Pacella, T. Weth,} Symmetry of solutions to semilinear elliptic equations via Morse index, {\em Proc. American Math. Soc.} {\bf 135} (2007), 1753-1762.

\bibitem[PT]{PT} {\sc J. Prajapat, G. Tarantello}, { On a class of elliptic problems in $\R^2$:
symmetry and uniqueness  results,  {\em Proc. Royal Soc. Edinburgh} {\bf 131}A,  (2001) 967-985.}

\bibitem[P]{Palais} {\sc R.S. Palais}, { The Principle o f Symmetric Criticality, {\em Commun. Math. Phys}
{\bf 69}, (1979),19-30.}

\bibitem[RW]{RW}{\sc X. Ren, J. Wei}, On a two-dimensional elliptic problem with large exponent in nonlinearity, {\em Trans. Amer. Math. Soc.} {\bf 343} (1994), 749-763.

\bibitem[ST]{SerraTilli}{\sc E. Serra, P. Tilli}  Monotonicity constraints and supercritical Neumann problems, 
{\em Ann. I. H. Poincar\'e AN} {\bf 28}, (2011), 63-74.


\bibitem[SSW]{SSW} {\sc D. Smets, M. Willem, J. Su}, Non-radial ground states for the H\'enon equation. {\em Commun. Contemp. Math.} {\bf 4}, (2002), 467-480.

\bibitem[ZY]{ZY} {\sc Y.-B. Zhang, H.-T. Yang}, 
Multi-peak nodal solutions for a two-dimensional elliptic problem with large exponent in weighted nonlinearity
{\em Acta Mathematicae Applicatae Sinica}, 31/1 (2015),  261-276. DOI: 10.1007/s10255-015-0465-5



\end{thebibliography}
\end{document}